\documentclass{amsart} 

\usepackage{etex}
\usepackage{lmodern}
\usepackage[T1]{fontenc}
\usepackage[utf8]{inputenc}

\usepackage{amsthm, amsmath, amssymb, amsfonts}
\usepackage{hyperref, cleveref}

\usepackage{mathrsfs}

\usepackage{mathtools} 

\usepackage{graphicx}
\usepackage{float, wrapfig}
\usepackage[all]{xy} 

\usepackage{xcolor}

\usepackage{listings}
\usepackage{enumerate}
\usepackage{multirow}

\theoremstyle{plain}
\newtheorem{Theorem}{Theorem}[section]
\newtheorem*{Theorem*}{Theorem}
\newtheorem*{Proposition*}{Proposition}
\newtheorem{Proposition}[Theorem]{Proposition}
\newtheorem{Lemma}[Theorem]{Lemma}

\newtheorem{Corollary}[Theorem]{Corollary}

\newtheorem*{Problem*}{Problem}

\theoremstyle{definition}
\newtheorem*{Example*}{Example}
\newtheorem*{Definition*}{Definition}

\theoremstyle{remark}
\newtheorem*{Remark}{Remark}

\newcommand{\Gset}{G\mathsf{set}}
\newcommand{\Rmod}{R\mathsf{-mod}}

\newcommand{\CC}{\mathbb{C}}

\newcommand{\FF}{\mathbb{F}}
\newcommand{\ZZ}{\mathbb{Z}}

\newcommand{\NN}{\mathbb{N}}
\newcommand{\KK}{\mathbb{K}}

\newcommand{\rr}{\mathcal{R}}
\newcommand{\nn}{\mathcal{N}}

\newcommand{\Fp}{\mathbb{F}_p}

\newcommand{\x}{^{\times}}

\newcommand{\Hom}{\operatorname{Hom}} 

\newcommand{\Mod}{\mathrm{mod} \ }
\DeclareMathOperator{\res}{Res}
\DeclareMathOperator{\ind}{Ind}
\DeclareMathOperator{\Ima}{Im}

\DeclareMathOperator{\st}{Stab}

\numberwithin{equation}{section} 
\newcommand\numberthis{\addtocounter{equation}{1}\tag{\theequation}} 

\title{Graded character rings, Mackey functors and Tambara functors}
\author{B\'{e}atrice I. Chetard}
\date{\today}

\begin{document}

	\begin{abstract}
	Let $G$ be a finite group and $\KK$ a field of characteristic zero. the ring $R_\KK(G)$ of virtual characters of $G$ over $\KK$ is naturally endowed with a so-called Grothendieck filtration, with associated graded ring $R^*_\KK(G)$. Restriction of representations to any $H\leq G$ induces a homomorphism $R^*_\KK(G) \to R^*_\KK(H)$. We show that, when $G$ is abelian, induction of representations preserves the filtration, so $R^*_\CC(-)$ is a Mackey functor; in the general case, we propose a modified filtration which turns $R^*_\KK(-)$ into a Mackey functor. We then turn to tensor induction of representations, and show that in the abelian case $R^*_\CC(-)$ is a Tambara functor.
	\end{abstract}

	\maketitle

	\section{Introduction}
Let $G$ be a finite group and $\KK$ a field of characteristic zero, and let $R_\KK(G)$ be the ring of virtual characters of $G$ over $\KK$. Exterior powers or representations turn $R_\KK(G)$ into a $\lambda$-ring, equipped with a so-called Grothendieck filtration. In \cite{chetard}, we undertook to compute examples of the graded character ring $R^*_\KK(G)$ associated with this filtration. The structures appearing as graded character rings of finite groups are remarkably complex; there is no known Künneth formula for $R^*_\KK(G)$, and computing graded rings of small groups reveals challenging. As an example, consider:
\[
  R^*_\CC(C_4\times C_4) = \frac{\ZZ[x,y]}{(4x,4y, 2x^2y + 2xy^2, x^4y^2 - x^2y^4)},
\]
where $|x| = |y| = 1$ (see \cite[Prop. 7.2]{chetard}). This result was obtained by combining the topological properties of the filtration with the functoriality of $R^*_\KK(-)$.\par

We turn here to a problem of a more abstract nature: graded character rings are functorial, and thus for any subgroup $H \leq G$, restriction of representations to $H$ induces a well-defined ring homomorphism $\res^G_H: R^*_\KK(G) \to R^*_\KK(H)$ (see \cite[Lem. 4.1]{chetard}). Does induction of representations yield a map $R^*_\KK(H) \to R^*_\KK(G)$? Such a transfer map would turn $R^*_\KK(-)$ into a \textbf{Mackey functor}, a particularly widespread type of algebraic structure: group cohomology, algebraic $K$-theory, character rings are all Mackey functors.
Among other results, the stable element method of Cartan and Eilenberg (see \cite{cartan-eilenberg}) generalises to all Mackey functors, and would allow us to relate the graded character ring of a group to those of its Sylow subgroups. A concise account of the general theory of Mackey functors is given in \cite{webb}.

If $S$ is any Mackey functor, then the following "stable element" result applies:
\begin{Proposition}\label{stableelementmethod}
  If $H := Syl_p(G)$ is abelian, then
  \[
    \res_H^G: S(G)_{(p)} \longrightarrow S(H)^{N_G(H)}
  \]
  is an isomorphism.
\end{Proposition}
Here $Syl_p(G)$ denotes a $p$-Sylow of $G$, and $S(H)^{N_G(H)}$ is the set of elements of $S(H)$ that are invariant under the action of the normalizer of $H$ in $G$. In the example of the alternating group $A_4$ of order $12$ (see \Cref{A4_gradrepring}), the surjectivity condition fails when restricting to the $2$-Sylow $C_4\times C_4$, and thus the graded character ring functor $R^*_\CC(-)$ is not a Mackey functor. This is \Cref{R*notmackey} in the text.

It is possible to "Mackeyfy" graded character rings by modifying the Grothendieck filtration. We define the \textbf{saturated filtration} $\{F^n(G)\}_{n\geq 0}$ as the minimal filtration that is preserved by induction of characters and contains the Grothendieck filtration, that is:
\[
  F^n(G) = \sum_{H\leq G}\ind_H^G(\Gamma^n(H))
\]
where $\{\Gamma^n(H)\}_{n\geq 0}$ is the Grothendieck filtration on $R_\KK(H)$. We call the associated graded ring the \textbf{saturated graded ring} of $G$ over $\KK$, and denote it $\rr^*_\KK(G)$ (as opposed to $R^*_\KK(G)$ for the usual graded ring).
Fortunately, restriction of representations also preserves this filtration, and thus:
\begin{Theorem}
  The saturated graded ring:
  \[
    \rr^*_\KK(-) := \bigoplus_{n\geq 0} F^n(-)/F^{n+1}(-)
  \]
  is a Mackey functor.
\end{Theorem}
(See \Cref{rr*mackey}). At a first glance, there is no guarantee that $\rr^*_\KK(-)$ is not trivial in some way or other: a lot of the information contained in the Grothendieck filtration could be lost in the process. Reassuringly, both filtrations contain the same information "at infinity", as we show in \Cref{gammatopology_coincides_ftopology}:
\begin{Theorem}
  The saturated filtration and the Grothendieck filtration induce the same topology on the character ring $R_\KK(G)$.
\end{Theorem}
This means, in particular, that induction of representations is continuous with respect to the Grothendieck topology, and can be extended to a map of completed rings $\widehat{\ind_H^G}: \widehat{R}_\KK(H) \to \widehat{R}_\KK(G)$. This, combined with the stable elements result, gives us \Cref{artinstheorem}, an analogue to Artin's theorem:
\begin{Theorem}
  Let $X$ be a family of subgroups of a finite group $G$. Let
	\[
		\widehat{\ind}: \bigoplus_{H \in X} \widehat{R}(H) \to \widehat{R}(G)
	\]
	be the morphism defined on each $\widehat{R}(H)$ by $\widehat{\ind_{H}^{G}}$.

  If $X$ contains a $p$-Sylow subgroup of $G$ for each prime $p$, then the map $\widehat{\ind}$ is surjective.
\end{Theorem}
Judging from the definition of the saturated filtration, one could expect $\rr^*_\KK(-)$ to remember some information about the subgroup structure of $G$. An intriguing open problem is whether $\rr^*(-)$ could distinguish groups with the same character table and power maps, something the usual graded character ring cannot do.

There are many examples of groups $G$ such that the two filtrations coincide, and the natural map $R^*_\KK(G) \to \rr^*_\KK(G)$ is an isomorphism (these two facts are actually equivalent, as we show in \Cref{saturatedgroupsequalfiltrations}); we call them \textbf{saturated groups}.
The following result combines \Cref{abeliangroupsaresaturated}, \Cref{Q8_is_saturated}, and \Cref{Dp_is_saturated}:
\begin{Theorem}
  Groups of order less than 12, as well as abelian groups, and dihedral groups of order $2p$ for $p$ prime, are saturated.
  In particular, the restriction of $R^*_\CC(-)$ to abelian groups is a Mackey functor.
\end{Theorem}
In general, the complexity of the saturated filtration makes direct computations difficult. However, since $\rr^*_\KK(-)$ is a Mackey functor, one can use the stable element method (\Cref{stableelementmethod} above), which allows us to determine the saturated ring of a group, knowing only those of its Sylow subgroups. This is the method we use for \Cref{rr*psl2p}:
\begin{Theorem}
  Let $G = PSL(2,p)$ be the projective special linear group over $\FF_p$, where $p$ is an odd prime such that $p \equiv 3,5 (\Mod 8)$. Write:
  \[
    |G| = 4\cdot p\cdot l_1^{i_1}\cdots l_n^{i_n}\cdot r_1^{j_1}\cdots r_m^{j_m}, \ \ \ \ \text{ with } \ \ l_k|(p-1), \ \  r_k|(p+1).
  \]
  Then:
  \begin{equation}
    \rr^*(G) \cong \frac{\ZZ[x_1,\cdots,x_n, y_1,\cdots y_m, z, t, u]}{(l_k^{i_k}x_k, r_k^{j_k}y_k, 2z,2t, pu, z^3 - t^2)}
  \end{equation}
  with $|x_k| = |y_k| = |z| = 2$, $|t| = 3$, $|u| = (p-1)/2$.
\end{Theorem}

The last problem we treat in this paper is that of tensor induction, a multiplicative map $R_\KK(H) \to R_\KK(G)$. Mackey functors equipped with such a multiplicative map (and satisfying certain axioms) are called \textbf{Tambara functors}. In group cohomology, this role is played by the Evens norm, wich, applied to the subgroup inclusion $G\hookrightarrow  G\times C_p$ (for $p$ prime), can be used to define Steenrod operations. The (ungraded) character ring $R_\KK(G)$ with tensor induction is also a Tambara functor, as we prove in \Cref{tambara}.

In order to explore the connection between tensor induction and the Grothendieck filtration, one needs to understand the behaviour of the multiplicative norm on virtual characters. This is a remarkably complex problem, as there is no known formula for the norm of the sum of two characters, even when those come from actual representations. We follow Tambara's account and, restricting first to normal subgroups of prime index, then to abelian groups, we obtain such a formula. This is the key to prove \Cref{abeliantambara}:
\begin{Theorem}
  The graded character ring functor $R^*_\CC(-)$ is a Tambara functor on abelian groups.
\end{Theorem}
As an application, we propose in \Cref{application} to compute, for any abelian group $G$, the norm of any degree-one Chern class from $R^*_\CC(G)$ to $R^*_\CC(G\times C_p)$.  This brings us one step closer to defining Steenrod operations on graded character rings.

The paper is organised as follows: we recall some definitions in \Cref{definitions}, where, in particular, we introduce the necessary notation to consider representation rings from the point of view of equivariant $K$-theory. In \Cref{stable_elements}, we show that graded character rings are not Mackey functors, through the example of the alternating group $A_4$; \Cref{saturatedrings} introduces the saturated filtration and explores its properties, which we apply in \Cref{computingsaturatedrings} to compute the saturated ring of, among others, the projective special linear group $PSL(2,p)$ for some primes $p$.
We then move on to Tambara functors in \Cref{tambara}, where we present a proof that equivariant $K$-theory is a Tambara functor. In \Cref{additionformula}, we study the norm of the sum of two characters, and determine a formula which shows that graded character rings of abelian groups are Tambara functors; we conlude in \Cref{application} by a straightforward application of these result to norms in abelian groups of the form $G\times C_p$.

\textbf{Acknowledgements.} I would like to thank both my Ph.D. advisors: Pierre Guillot, for his wise mentorship and his uncountable encouragements, and J\'an Min\'a\v{c} for his unfaltering enthusiasm.


\section{Definitions and notations}\label{definitions}
We briefly recall the definition of a graded character ring. For details, we refer the reader to \cite{chetard}.

Let $G$ be a finite group and $\KK$ a field of characteristic zero. The \textit{ring of virtual characters} (or character ring) $R_\KK(G)$ over $\KK$ is the Grothendieck ring of the category of $\KK G$-modules; that is, it is the abelian group generated by irreducible representations of $G$ up to isomorphism, with multiplication given by the tensor product of representations. Since $\KK$ has characteristic zero, representations up to isomorphism are fully determined by their character, thus we'll use both terms interchangeably. The character ring is an augmented ring, with augmentation $\epsilon: R_\KK(G) \to \ZZ$ sending a character to its degree.
Additionally, exterior powers of representations turn $R_\KK(G)$ into a $\lambda$-ring: if $\chi$ is a character of $G$ over $\KK$ afforded by a representation $\rho$, we write $\lambda^n(\chi)$ for the character of the $n$-th exterior power $\Lambda^n \rho$ of $\rho$. The operations $\{\lambda^n\}$ can be extended to the whole ring $R_\KK(G)$, and satisfy the axioms for a $\lambda$-ring, which are detailed in \cite{atiyah-tall}. For $x \in R_\KK(G)$ and $n \in \NN$ put
\[
	\gamma^n(x) = \lambda^n(x +n -1),
\]
the $n$-th gamma operation. Let $I = \ker \epsilon$ be the augmentation ideal. We define the $n$-th ideal $\Gamma^n$ in the Grothendieck filtration (or $\Gamma$-filtration) as the abelian subgroup generated by monomials of the form
\[
	\gamma^{n_1}(x_1)\cdot\gamma^{n_2}(x_2)\cdots\gamma^{n_k}(x_k), \ x_i\in I, \ \sum_{i = 1}^k n_i \geq n.
\]
Then $\Gamma^n$ is a $\lambda$-ideal (that is, an ideal that is preserved by $\lambda$-operations) in $R_\KK(G)$. We have $\Gamma^0 = R_\KK(G)$ and $\Gamma^1 = I$. Moreover, $\Gamma^n\cdot\Gamma^m \subseteq \Gamma^{n+m}$, therefore we can define the \textit{graded character ring} of $G$ over $\KK$ as follows:
\[
	R^*_\KK(G) = \bigoplus_{i\geq 0} \Gamma^{i}/\Gamma^{i+1}.
\]
From now on, we write $R(G)$ and $R^*(G)$ when $\KK$ is clear from the context. Note that although most of our general results are independent of the field $\KK$, we always compute explicit examples over $\KK = \CC$. The ring $R^*(G)$ is generated by \textit{Chern classes} of irreducible representations; for the full definition and properties of Chern classes, we refer the reader to \cite{chetard}. Suffice it to say that the \textit{$n$-th algebraic Chern class} $c_n(\rho)$ of a character $\rho$ is defined as the image in $R^*(G)$ of the element
\[
	C_n(\rho) = \gamma^n(\rho - \epsilon(\rho)).
\]
It is an element of degree $n$ in $R^*(G)$.

The definitions of Mackey and Tambara functors, to be given in later sections, are greatly simplified by looking at character rings from the point of view of $G$-equivariant $K$-theory. We view a $G$-set $X$ as a category with an object for each point, and an arrow between two objects $(g,x): x \to y$ for each $g \in G$ such that $g\cdot x = y$.
A vector bundle is then defined as a functor $V$ between $X$ and the category of $\KK$-vector spaces and linear maps; that is, it associates to each $x\in X$ a vector space $V_x$, and to each $g\in G$ linear maps $V_{(g,x)}: V_x \to V_{g\cdot x}$. For an element $e \in V_x$, we write $g  \cdot e \in V_{g \cdot x}$ for $V_{(g,x)}e$.
A functor $V$ then corresponds to the data of each $V_x$ and $g\cdot e$. Let $K_G^+(X)$ be the semigroup of isomorphism classes of vector bundles over $X$ under direct sum. In the sequel, we restrict ourselves to finite $G$-sets.
	\begin{Lemma}
		Let $X$ be a transitive $G$-set with a distinguished point $x \in X$, and let $H = \mathrm{Stab}(x)$. Then there is an isomorphism (depending on $x$) between $K_G^+(X)$ and the semiring of representations $R^+(H)$.
	\end{Lemma}
	\begin{proof}
		Let $W$ be a representation of $H$ and consider the induced representation $V = \ind_{H}^{G} = \KK[G] \otimes_\KK[H] W$. Define a vector bundle on $X$ as follows: for each $y \in X$, write $y = g\cdot x$ and let $(V)_y = g\cdot W = g\otimes W \subset V$. This depends only on $y$ and the action of $g$ takes $V_{x}$ to $V_{gx}$, so this is a vector bundle.
		Conversely, if $V$ is a vector bundle on $X$, define $W = V_{x}$. This is a well-defined $H$-module (since it is stable by $H$), so $W$ is a representation. These two constructions are mutually inverse.
	\end{proof}
	\begin{Remark} \begin{enumerate}
		\item The isomorphism above depends on $x$; choosing the point $y = g\cdot x$ as a basepoint instead, one obtains the isomorphic representation of $gHg^{-1}$ which is given by precomposing the action of $H$ on $V_x$ by conjugation with $g$.
		\item As a direct corollary, the Grothendieck group $K_G(X)$ of vector bundles and the ring $R(G)$ are isomorphic.
		\item Since every finite $G$-set can be written as a disjoint union of transitive $G$-set, this gives us a way to prove general facts about $K_G^+(X)$ by restricting to representation rings.
	\end{enumerate}

	\end{Remark}
	This vocabulary allows us to generalize the notions of restriction, transfer and tensor induction of representations. Let $f : X \rightarrow Y$ be a map of $G$-sets, given by a functor between the categories $X$ and $Y$ as described above. We define:
	\begin{enumerate}
		\item The \textit{restriction} $f^*: K_G^+(Y) \rightarrow K_G^+(X)$, as the composition of $f$ and $V$. In other words:
		\begin{align*}
			\left(f^*V\right)_{x} &:= V_{f(x)}, \\
			(f^*V)_{(g,x)} &:= V_{(g, f(x))}.
		\end{align*}
		Note that with the shorthand notation mentioned above, for $e \in V_{f(x)}$, we have $(f^*V)_{(g,x)}e = g\cdot e$, which corresponds to the same element in $f^*(V)$ as in $V$, only understood in a different fibre. This is particularly intuitive in the case where $f: G/H \to G/K$ corresponds to the inclusion of a subgroup $H \hookrightarrow K$.
		\item The \textit{induction} (or transfer) $f_*: K_G^+(X) \rightarrow K_G^+(Y)$,
		\begin{align*}
			f_*(V)_y &:= \bigoplus_{x \in f^{-1}(y)}(V_x), \\
			f_*(V)_{(g,y)} &:= \bigoplus_{x \in f^{-1}(y)} V_{(g,x)}.
		\end{align*}
		In shorthand notation we have $g \cdot \left( \bigoplus_x e_x\right) = \bigoplus_x g\cdot e_{g^{-1}x}$.
		\item The \textit{norm} (or tensor induction) $f: K_G^+(X) \rightarrow K_G^+(Y)$,
		\begin{align*}
			f_\sharp(V)_y &:= \bigotimes_{x \in f^{-1}(y)} V_x, \\
			f_\sharp V_{(g,y)} &:= \bigotimes_{x \in f^{-1}(y)}V_{(g,x)}.
		\end{align*}
		With $g\cdot \left(\bigotimes_x e_x \right) = \bigotimes_x g\cdot e_{g^{-1} x}$.
	\end{enumerate}
	Note that although we use its vocabulary and definitions, the full extent of equivariant $K$-theory is beyond our scope. Thus we will mostly assume that $X,Y$ are of the form $G/K$ for some subgroup $K\leq G$, and more often than not we will have $Y = G/G = \{*\}$.
	\begin{Remark}
		\begin{enumerate}
			\item Equivalently,
				\[
					f_*(V)_y \cong \bigoplus_{x \in f^{-1}(y) / \st(y)} \ind_{\st(x)}^{\st(y)}V_x.
				\]
			\item One can check that applying the norm formula to the case of $X = G/H$ and $Y = \{*\}$ yields the usual tensor induction, as defined eg. in \cite[\S 13A]{curtis}
		\end{enumerate}
	\end{Remark}
	The restriction and induction maps can be extended to $K_G$ in a straightforward way. By \cite[Th. 6.1]{tambara}, so can the tensor induction map. We explore in \Cref{additionformula} how to determine a formula for the tensor induction of virtual characters.

	\section{Graded character rings are not Mackey functors} \label{stable_elements}
	Graded character rings are functorial (see \cite[Lemma 4.1]{chetard}); in particular, if $H$ is a subgroup of $G$, restricting representations from $G$ to $H$ induces a well-defined homomorphism $R^*(G) \to R^*(H)$. Naturally, one wonders whether induction of representations from a subgroup $H$ of $G$ also preserves the Grothendieck filtration, and thus gives rise to a well-defined, additive induction map from $R^*(H)$ to $R^*(G)$. If so, then $R^*(-)$ satisfies the axioms of a cohomological Mackey functor (which we define below). In particular, an analogue to Cartan and Eilenberg's result on stable elements in cohomology (\cite[Th. XII.10.1]{cartan-eilenberg}) states that each $p$-primary component $R^*(G)_p$ of $R^*(G)$ is isomorphic to some subring of the graded character ring of its $p$-Sylow subgroup. This is not the case, and we produce below an example where this property fails. Thus $R^*(-)$ cannot be a Mackey functor.

A thorough treatment of the theory of Mackey functors is given in \cite{webb}; let us start with the definition. Let $R$ be a commutative ring and $G$ a group, and let $\Gset$ be the category of finite $G$-sets. A \textit{Mackey functor} is a pair $(S^*, S_*)$ of functors from $\Gset$ to $\Rmod$, where $S^*$ is contravariant and $S_*$ is covariant, and $S^*(-)$ and $S_*(-)$ are equal on objects. Additionally, we require the following axioms be satisfied:
\begin{enumerate}
	\item If
		\[ \xymatrix{ \Omega_1 \ar[r]^\alpha \ar[d]_\beta & \Omega_2 \ar[d]^\gamma \\
			\Omega_3 \ar[r]_\delta & \Omega_4 }
		\]
		is a pullback diagram of $G$-sets, then $S^*(\delta)S_*(\gamma) = S_*(\beta)S^*(\alpha)$.
	\item For every pair $\Omega, \Psi$ of finite $G$-sets, the morphism $S(\Omega)\oplus S(\Psi) \to S(\Omega\sqcup\Psi)$ obtained by applying $S_*$ to $\Omega \to \Omega\sqcup\Psi \leftarrow \Psi$, is an isomorphism.
\end{enumerate}


\begin{Remark}
	Alternatively, a Mackey functor $S$ can be viewed a a function from the subgroups of $G$ to $\Rmod$, with, for any two subgroups $H \leq K$ and $g\in G$, maps $\res^{H}_{K}: S(H) \to S(K)$, $\ind_{K}^{H}: S(K) \to S(H)$ and $c_g: S(H) \to S(^gH)$. The maps are required to satisfy the usual axioms governing conjugation, induction and restriction of representations, as detailed in \cite[\S 2]{webb}. If, additionally, the induction and restriction satisfy $\res_K^H(\ind_ K^H(x)) = [K:H]\cdot x$ then $S(-)$ is called a cohomological Mackey functor.
\end{Remark}
This second definition makes it easy to check that the (ungraded) character ring $R(-)$ is a cohomological Mackey functor. Thus, if induction preserves the filtration, then $R^*(-)$ is also a cohomological Mackey functor. The following general result should then be valid for $R^*(-)$.
\begin{Proposition}[Cartan-Eilenberg]
	Let $H\geq G$ be any subgroup, and suppose $S$ is a Mackey functor. Call an element $x \in S(H)$ stable if
	\[
		\res^{^gH}_{^gH\cap H}(c_g(x)) = \res^H_{^gH\cap H}(x)
	\]
	for all $g\in G$. If $G \geq H \geq Syl_p(G)$ where $Syl_p(G)$ is a $p$-Sylow of $G$, and $S(G)_{(p)}$ denotes the $p$-primary component of $S(G)$, then:
	\[
			\res_H^G : S(G)_{(p)} \longrightarrow S(H)_{(p)}
	\]
	is injective, and its image consists of the stable elements in $S(H)_{(p)}$.
\end{Proposition}
	This result is a consequence of \cite[Cor. 3.7 and Prop. 7.2]{webb}; a more elementary proof, in the case of cohomology, can be found in \cite[Th. 6.6]{adem-milgram}. In the case of the alternating group $A_4$ of order $12$, we use the following corollary:
	\begin{Corollary}
		If $H := Syl_p(G)$ is abelian, then
		\[
			\res_H^G: \rr^*(G)_{(p)} \longrightarrow \rr^*(H)^{N_G(H)}
		\]
	is an isomorphism.
	\end{Corollary}
	 We show that the condition of surjectivity on stable elements fails. Note that the following computation relies heavily on the techniques developed in \cite{chetard}, to which we refer the reader for any details. We also use the following result:
	\begin{Lemma}[{\cite[Prop. 4.3]{chetard}}] \label{R*(C2C2)}
		Let $C_2$ be the cyclic group of order $2$, and let $\rho_1, \rho_2$ be the generating representations for $R_\CC(C_2\times 1), R_\CC(1\times C_2)$ respectively. Then,
		\[
			R^*_\CC(C_2\times C_2) = \frac{\ZZ[t_1,t_2]}{(2t_1,2t_2,t_1^2t_2-t_1t_2^2)}
		\]
		where $t_i = c_1(\rho_i)$.
	\end{Lemma}
	Let $A_4$ be generated by the permutations $(12)(34)$ and $(123)$. There are $4$ irreducible complex representations of $A_4$:
	\begin{itemize}
		\item Of dimension 1: the trivial representation $1$, and the representations $\rho$ (resp. $\bar{\rho}$) that send $(123)$ to $e^{2i\pi/3}$ (resp. $e^{-2i\pi/3}$) and $(12)(34)$ to $1$.
		\item Of dimension 3: the standard representation $\theta$, which is the quotient of the representation $\bar{\theta}$ acting on $\CC^4$ by permutation of the basis vectors, by the trivial representation. The character of $\theta$ sends 3-cycles to $0$ and $(12)(34)$ to $-1$.
	\end{itemize}
	There are the following relations between the representations:
	\begin{align}
		\rho^2 &= \bar{\rho} \\
		\rho\theta &= \theta \\
		\theta^2 &= 1+\rho+\bar{\rho}+2\theta \label{A4_relationtheta2}
	\end{align}
	Additionally $\lambda^2(\theta) = \theta$ (by a direct calculation of the exterior power) and $\mathrm{det }(\theta) = 1$.
	\begin{Lemma} \label{A4_gradrepring}
		Let $x = c_1(\rho)$ and $y = c_2(\theta)$, then
		\[
			R^*_\CC(A_4) = \frac{\ZZ[x,y]}{(3x,12y,4y+x^2)}.
		\]
	\end{Lemma}

	\begin{proof} The graded character ring $R^*_\CC(A_4)$ is generated by all Chern classes of irreducible characters of $A_4$; we start by ridding ourselves of redundant generators. Let $x = c_1(\rho)$ and $y = c_2(\theta)$. Then $x = c_1(\rho) = -c_1(\bar{\rho})$ and $3x = 0$. Moreover $c_1(\theta) = c_1(\mathrm{det } \theta) = c_1(1) = 0$, so $x$ generates $R^1(A_4)$, and $y, x^2$ generate $R^2(A_4)$. As for the degree $3$ generator $c_3(\theta)$, we have:
		\[
			C_3(\theta) = \gamma^3(\theta-3) = \lambda^3(\theta-1) = -1 +\theta -\lambda^2(\theta) + 1 = 0,
		\]
		so there is no additional generator in degree 3 and $R^*(A_4)$ is generated by $x,y$. We have $3x = 0$ by the above, and $12y = 0$ since the order of $A_4$ kills $R^*(A_4)$ (see \cite[Prop. 2.6]{chetard}). We now turn to the relation $4y+x^2 = 0$: applying the total Chern class $c_T$ to both sides of (\ref{A4_relationtheta2}) yields:
		\begin{align*}
			c_T(\theta^2) &= c_T(1+\rho+\bar{\rho}+2\theta) \\
			&= c_T(\rho)c_T(\bar{\rho})c_T(\theta)^2 \\
			&= (1+xT)(1-xT)(1+yT^2)^2 \\
			&= 1+ (2y-x^2)T^2 + (y^2-2yx^2)T^4 + zy^2T^6.	\numberthis	\label{A4_totalchern1}
		\end{align*}
		On the left-hand side, use the splitting principle (\cite[Prop. 2.3]{chetard}): we write the character $\theta$ as a sum $\theta_1+\theta_2+\theta_3$ of linear characters. Looking only at even terms of degree $\leq 6$ and keeping in mind that $c_1(\theta) = c_3(\theta) = 0$, we get:
		\begin{align*}
			c_T(\theta^2) &= c_T((\sigma_1+\sigma_2+\sigma_3)^2) \\
				&= 1 + 6yT^2 + 9y^2T^4 + 4y^3T^6 \numberthis \label{A4_totalchern2}
		\end{align*}
	  Equating \ref{A4_totalchern1} and \ref{A4_totalchern2} yields $4y = -x^2$. In particular this means that the order of $y$ is a multiple of $3$. To obtain more information, we can use the restriction $\res^{A_4}_{C_2\times C_2}: R^*(A_4) \to R^*(C_2\times C_2)$ $y$ to $H:= C_2\times C_2$. By \Cref{R*(C2C2)} :
		\[
			R^*(\ZZ/2\times \ZZ/2) = \frac{\ZZ[t_1,t_2]}{(2t_1,2t_2,t_1^2t_2-t_1t_2^2)}.
		\]
		We have $\res_H(y) = t_1^2 + t_1t_2 + t_2^2$, which has order $2$. So the order of $y^i$ is a multiple of $2$, that is, it is either $6$ or $12$. To conclude, we use the continuity method described in (\cite[\S 6]{chetard}). Let $X = C_1(\rho) = \rho -1$ and $Y = C_2(\theta) = 3 - \theta$, and let
		\[
			\widetilde{\Gamma}^n = \begin{cases} \langle Y^{n/2}\rangle &n \text{ even} \\
				\langle XY^{(n-1)/2}\rangle &n \text{ odd}
			\end{cases} \ \ \ .
		\]
		Then $\widetilde{\Gamma}^n$ is an admissible approximation for $\Gamma$. The evaluation $\phi_{(12)(34)}$ sends $X$ to $0$ and $Y$ to $-4$, and thus is continuous with respect to the $2$-adic topology on $\ZZ$. Suppose, for a contradiction, that $6Y^k \in \Gamma^{2k+1} = \widetilde{\Gamma}^{2k+1} + \Gamma^M$ for some large $M$. Since $2k+1$ is odd, $\widetilde{\Gamma}^{2k+1}$ is generated by $XY^k$, which evaluates to zero. Thus we must have $v_2(6Y^k) \geq 2k+2$, but $v_2(6Y^k) = 2k+1$. So $Y^k$ has additive order $12$.
		Finally, restricting $x$ to the subgroup generated by $(123)$, and $yx$ to that generated by $(12)(34)$, shows that there are no additional relations.
	\end{proof}

	\begin{Theorem} \label{R*notmackey}
		$R^*(-)$ is not a Mackey functor.
	\end{Theorem}
	\begin{proof}
		Let $G = A_4$, and consider its normal, abelian $2$-Sylow $H = C_2 \times C_2$. Since $\res^{G}_{H}(x) = 0$, the image of $R_\CC^*(G)$ under the restriction map
		\[
			\res_H^G: R^*_\CC(G) \longrightarrow  R^*_\CC(H)
		\]
		is generated by powers of $\res_H^G(y) = t_1^2 + t_1t_2 + t_2^2$. On the other hand, $G$ acts on $R^*_\CC(H)$ by cyclic permutations of the elements $t_1, t_2,t_1+t_2$. The element $z = t_1^3 + t_2^3+ t_1^2t_2$ is invariant under this action. But $z$ is not a combination of powers of $t_1^2 + t_1t_2 + t_2^2$ since it has odd degree, and thus does not belong to the image of the restriction map. Therefore:
		\[
			\Ima(\res_H^G) \subsetneq R^*_\CC(H)^{N_G(H)}
		\]
		which means that $R^*(-)$ is not a Mackey functor.
	\end{proof}


\section{Saturated rings} \label{saturatedrings}
\Cref{R*notmackey} tells us that induction of representations is not compatible with the Grothendieck filtration. This prompts us to define a modified filtration, taking into account all images of Chern classes of subgroups of $G$ under the induction map. This new filtration retains much of the information of the Grothendieck filtration: in fact, both induce the same topology on $R(G)$. In the sequel, let $H,K$ denote two arbitrary subgroups of $G$.
On the $\lambda$-ring $R(G)$, define the \textit{saturated filtration} $\lbrace F^n \rbrace_{n}$ as follows:
\[
	F^n(G) = \sum_{H\leq G} \ind_H^G(\Gamma^n(H)).
\]
This means that $F^n(G)$ is generated by elements of the form:
\[
	x = \ind_{H}^G(\gamma^{i_1}(\rho_{1})\cdots \gamma^{i_m}(\rho_{m})), \ \ \ i_1+\cdots+i_m \geq n
\]
with each $\rho_\ell$ an irreducible representation of $H$. By definition, induction of representations preserves the filtration $F$.
\begin{Lemma}\label{saturatedfiltration} Let $I = \ker \epsilon$ be the augmentation ideal.
	\begin{enumerate}
		\item Induction and restriction of characters preserve the filtration $F$.
		\item $F^i(G)\cdot F^j(G) \subseteq F^{i+j}(G)$.
		\item $F^0(G) = R(G), \ F^1(G) = I$.
	\end{enumerate}
\end{Lemma}
\begin{proof}
	\begin{enumerate}
	 	\item We only need to check that restriction does preserve the filtration. Let $x\in F^i(G)$; we prove that if $x = \ind_H^G(y)$ with $y \in \Gamma^i(H)$ then $\res^G_H(x) \in F^i(K)$ for all $K \leq G$. We use the Mackey double coset formula (\cite[Prop. 7.3.22]{serre}): let $S$ be a set of $(H,K)$-double coset representatives of $G$. For $s\in S$, let
	 	\[
	 		{}_sH = sHs^{-1}\cap K \leq K
	 	\]
	 	and let
	 	\[
	 		y^s(g) = y(s^{-1}gs), \ \ \ \text{for all} \ \ \ g \in {}_sH.
	 	\]
	 	Then $y^s$ is a representation of ${}_sH$ and
	 	\[
	 		\res^G_K\ind_H^G(y) = \sum_{s\in K \backslash G / H} \ind_{{}_sH}^K(y^s)
	 	\]
	 	Note that each $y^s$ is in $\Gamma^i(^sH)$ by functoriality of $R^*(-)$. Thus $\res^G_K(x) \in F^i(K)$.

	 	\item It is sufficient to prove that if $\tilde{x} = \ind_H^G(x)$ and $\tilde{y} = \ind_K^G(y)$ with $x\in \Gamma^i(H)$ and $y \in \Gamma^j(K)$ then $\tilde{x}\cdot\tilde{y} \in F^{i+j}(G)$. We proceed by induction on the order of $G$. Suppose $H < G$ is a proper subgroup. By the projection formula (\cite[\S 7.2]{serre}):
	 	\[
	 		\tilde{x}\cdot\tilde{y} = \ind_H^G(x)\ind_K^G(y) = \ind_H^G(x\cdot\res^G_H\ind_K^G(y)).
	 	\]
	 	Since restriction preserves the filtration and $\ind_K^G(y) \in F^j(G)$, we have $\res^G_H\ind_K^G(y) \in F^j(H)$, so:
	 	\[
	 		x\cdot\res^G_H\ind_K^G(y) \in F^i(H)F^j(H) \subseteq F^{i+j}(H).
	 	\]
	 	where the inclusion is true by induction. In conclusion:
	 	\[
	 		\tilde{x}\cdot\tilde{y} = \ind_H^G(x\cdot\res^G_H\ind_K^G(y)) \in F^{i+j}(G).
	 	\]

	 	\item The fact that $F^0(G) = R(G)$ comes from the fact that $\Gamma^0(G) = R(G)$. For $F^1(G)$, simply observe that, since $\varepsilon(\ind_H^G(\rho)) = [G:H]\varepsilon(\rho)$:
	 	\[
			F^1(G) = \sum_{H \leq G} \ind_H^G(\ker(\varepsilon\vert_H)) = \ker(\varepsilon) = I.
		\]
	\end{enumerate}
\end{proof}

\Cref{saturatedfiltration} lets us define the \textit{saturated graded ring} associated to $G$ as:
	\[
		\rr^*(G) = \bigoplus_{i\geq 0}F^i(G)/F^{i+1}(G).
	\]
	Note that, as representation rings are of the form $K_G(X)$ for some transitive $G$-set $X$, we can extend the definition of this filtration to $K_G(X)$ for a general finite $G$-set $X$. Then the above discussion means that for every maps of finite $G$-sets $f: X \to Y$, the maps $f_*$ and $f^*$ defined in \Cref{definitions} are compatible with the saturated filtration.

\begin{Theorem}\label{rr*mackey}
	The saturated graded ring $\rr^*(-): \Gset \to \ZZ-\Mod$ is a Mackey functor.
\end{Theorem}
\begin{proof}
This follows from the above.
\end{proof}

Note that $\rr^*$ is actually a Green functor, that is, a Mackey functor with an $R$-algebra structure compatible with restriction and satisfying the projection formula. We still need to ensure we do not lose too much information by modifying the filtration: after all, we could end up with trivial graded rings. It is not the case however, and in fact both filtrations induce the same topology on $R(-)$. We rely on the following result by Atiyah:

\begin{Lemma}[{\cite[Cor. 12.3]{atiyah-characters_cohomology}}]\label{gammatopology_coincides_idaictopology}
The topology induced by the Grothendieck filtration coincides with the $I$-adic topology.
\end{Lemma}

\begin{Theorem} \label{gammatopology_coincides_ftopology}
	The filtrations $(F^n)_n$ and $(\Gamma^n)_n$ induce the same topology on $R(G)$.
\end{Theorem}
\begin{proof}
 	Let $U\subseteq R(G)$ be open for the $F$-topology, that is, for any $x\in U$ there is an integer $N$ such that $x+F^N \subseteq U$. Since $\Gamma^N \subseteq F^N$, we also have $x+\Gamma^N \subseteq U$, so $U$ is also open for the $\Gamma$-topology.\\
	To prove that a set $U$ open in the $\Gamma$-topology is also open in the $F$-topology, we need to show that for each $N$, there is an $M$ such that $F^M \subseteq \Gamma^N$. Let $H \leq G$, and recall that $R(H)$ can be viewed as an $R(G)$-module via the restriction homomorphism. Then, by \cite[Th. 6.1]{atiyah-characters_cohomology}, the $I(H)$-adic topology is equal to the topology on $R(H)$ induced by the $I(G)$-adic topology. By \cref{gammatopology_coincides_idaictopology}, these topologies are also equal to the $\Gamma$-topology on $H$. In particular, for our fixed $N$, there are some $k$,$m$ satisfying:
	\[
		\Gamma^N(H) \supset I(G)^k\cdot R(H) \supset \Gamma^m(H).
	\]
	Pick $k$ (and thus $m$) large enough that we also have $I(G)^k \subset \Gamma^N(G)$. Then
	\[
		\ind_H^G(\Gamma^m(H)) \subset \ind_H^G(I(G)^k\cdot R(H)) \subset I(G)^k \subset \Gamma^N(G).
	\]
	Now let
	\[
		M = \max_{H\leq G} \ \left\lbrace \min \  \{m \ \vert \ \ind_H^G(\Gamma^m(H)) \subset \Gamma^N(G)\}\right\rbrace,
	\]
	then $F^M(G) = \sum_{H\leq G}\ind_H^G(\Gamma^M(H)) \subset \Gamma^N(G)$, which completes the proof.
\end{proof}

Since $\Gamma^n \subseteq F^n$ for all $n \geq 0$, there is a natural map of graded rings:
\[
	\eta: R^*(G) \longrightarrow \rr^*(G)
\]
induced by the identity. Here is a neat consequence of \cref{gammatopology_coincides_ftopology}:

\begin{Corollary}\label{saturatedgroupsequalfiltrations}
 If the natural map $\eta: R^*(G) \rightarrow \rr^*(G)$ is surjective, then it is an isomorphism and the filtrations $(F^n)$ and $(\Gamma^n)$ are equal.
\end{Corollary}
\begin{proof}
	If $\eta$ is surjective, then $\rr^*(G)$ is generated by Chern classes of elements of $R(G)$. Let $P_{w}$ denote a polynomial in the $C_l(\rho_k)$ of weight $w$, then any $x\in F^n(G)$ can be written as:
	\[
		x 	= P_{n}(C_{i_1}(\rho_{1}), \cdots, C_{i_k}(\rho_k)) + y_{n+1}
	\]
	where the $\rho_j$'s are irreducible representations of $G$ and $y_n \in F^{n+1}$. Then we also have:
	\begin{align*}
		x 	&= P_{n}(C_{i_1}(\rho_{1}), \cdots, C_{i_k}(\rho_k)) + P_{n+1}(C_{i_1}(\rho_{1}), \cdots, C_{i_k}(\rho_k)) + y_{n+2} \\
			&= \sum_{l = 1}^m P_{n+l}(C_{i_1}(\rho_{1}), \cdots, C_{i_k}(\rho_k)) + y_{m+1},
	\end{align*}
	for any positive $m$. So $x$ is in $\Gamma^n(G) + F^m(G)$ for all $m$, that is, $x$ is in the topological closure $\overline{\Gamma^n(G)}$ of $\Gamma^n(G)$. But $\Gamma^n(G)$ is closed in $R(G)$, thus $x \in \Gamma^n(G)$.
\end{proof}

We say that $R^*(G)$ is \textit{saturated} if the natural map $\eta$ is an isomorphism. A group $G$ is \textit{saturated} (over $\KK$) if $R^*_\KK(G)$ is saturated. For $H \leq G$, if the induction $i_*:R(H) \rightarrow R(G)$ is compatible with the filtration $(\Gamma^n)$, then $H$ is said to be \textit{$\Gamma$-compatible} with $G$.

\begin{Lemma} \label{restrictionsurjective}
	If the restriction maps $i^*: R(G) \rightarrow R(H)$ are surjective for all $H\leq G$, then $G$ is saturated.
\end{Lemma}
\begin{proof}
	First note that if $i^*$ is surjective then each $i^*_{\Gamma^M}:\Gamma^M(G) \to \Gamma^M(H)$ is surjective: a monomial in the $\gamma^l(\rho_k)$, with $\rho_k\in R(H)$ is just the image by $i^*$ of $\gamma^l(\sigma_k)$, with $\sigma_i\in R(G)$ satisfying $i^*(\sigma_k) = \rho_k$. So let $\rho \in \Gamma^M(H)$ and pick some $\sigma \in \Gamma^M(G)$ such that $\rho = i^*(\sigma)$. Then:
		\[
		i_*(\rho) = i_*i^*(\sigma) = i_*(1)\sigma \in \Gamma^M(G).
		\]
	So all virtual characters in $F^n(G)$ (which are induced from subgroups of $G$) are also in $\Gamma^n(G)$, and thus $\rr^*(G) = R^*(G)$.
\end{proof}
\begin{Remark}
		\begin{enumerate}
			\item In \Cref{saturatedrings}, we use \Cref{restrictionsurjective} to show that Abelian groups are saturated. So $R^*(-)$ is a Mackey functor when restricted to abelian groups.
			\item We show in \Cref{Dp_is_saturated} that the converse of \Cref{restrictionsurjective} is not true: the dihedral group of order $D_p$ for $p$ odd is saturated, but restriction of representations to $C_p$ isn't surjective.
		\end{enumerate}
\end{Remark}

The following result implies that the saturated graded ring of $G$ is completely determined by that of its Sylow subgroups. It is a consequence of \cite[Cor. 3.7 and Prop. 7.2]{webb}; for a more concrete proof, see for example \cite[Th. 6.6]{adem-milgram}.

\begin{Theorem} \label{theorem_swan}
	Let $G \geq H \geq Syl_p(G)$ where $Syl_p(G)$ is a $p$-Sylow of $G$ and let $\rr^*(G)_{(p)}$ denote the $p$-primary component of $\rr^*(G)$. Then:
	\[
		\res_H^G : \rr^*(G)_{(p)} \longrightarrow \rr^*(H)_{(p)}
	\]
	is injective, and its image consists of the stable elements in $\rr^*(H)_{(p)}$
\end{Theorem}

A similar result to that due to Swan in cohomology (see \cite{swan}) can be obtained as a straightforward application of \Cref{theorem_swan}.

\begin{Corollary}[Swan's Lemma]
	If $H \unlhd G$ is a normal subgroup such that $H \supseteq Syl_p(G)$, then
	\[
		\rr^*(G)_{(p)} \cong \Ima(\res_H^G) = \rr^*(H)^{G/H}_{(p)}
	\]
\end{Corollary}
\begin{proof}
	If $H$ is normal, the stability condition becomes $c_g(x) = x$, that is, $x$ is invariant by the action of $G/H$.
\end{proof}

\begin{Corollary} \label{swanslemma}
	If $H := Syl_p(G)$ is abelian, then
	\[
		\res_H^G: \rr^*(G)_{(p)} \longrightarrow \rr^*(H)^{N_G(H)}
	\]
is an isomorphism.
\end{Corollary}
\begin{proof}
	See \cite[Th 6.8]{adem-milgram}.
\end{proof}

\begin{Corollary}\label{corestrictionissurjective}
	Let $H = Syl_p(G)$ be a $p$-Sylow subgroup. Then the induction map
	\[
		\ind_{H}^{G}: \rr^*(H) \to \rr^*(G)_{(p)}
	\]
	is surjective.
\end{Corollary}
\begin{proof}
 First note that since $\rr^*(H)$ is $p$-torsion, the image of $\ind_{H}^{G}$ is indeed contained in $\rr^*(G)_{(p)}$. Pick an element $x\in \rr^*(G)_{(p)}$, then $\ind_{H}^{G}\res^{G}_{H}(x) = [G:H]x$, and $[G:H]$ is invertible in $\rr^*(G)_{(p)}$.
\end{proof}

Note that since the induction map preserves the $F$-filtration, it is continuous with respect to the topology induced by it (and thus with respect to the $\Gamma$ and $I$-adic topologies). In particular, induction extends to a well-defined map of completed rings
\[
	\widehat{\ind_{H}^{G}}: \widehat{R}(H) \to \widehat{R}(G)
\]
and by \cref{corestrictionissurjective} the characters induced from Sylow subgroups of $G$ form a dense subset of the completed ring $\widehat{R}(G)$. In other words, we have the following variant of Artin's theorem (see \cite[Th. II.9.17]{serre}):
\begin{Theorem} \label{artinstheorem}
	Let $X$ be a family of subgroups of a finite group $G$. Let
	\[
		\widehat{\ind}: \bigoplus_{H \in X} \widehat{R}(H) \to \widehat{R}(G)
	\]
	be the morphism defined on each $\widehat{R}(H)$ by $\widehat{\ind_{H}^{G}}$. If $X$ contains a $p$-Sylow of $G$ for all $p$, then the map $\widehat{\ind}$ is surjective.
\end{Theorem}
\begin{proof}
	By \Cref{corestrictionissurjective}, the characters induced from $H_p$ form a dense subset of $\widehat{R}(G)_{(p)}$ for the $F$-topology, so if $X$ contains a $p$-Sylow of $G$ for every $p$ then $\widehat{\ind}$ is surjective.

\end{proof}


\section{Computing saturated rings}\label{computingsaturatedrings}
We now apply \Cref{saturatedrings} by trying our hand at some computations; a number of the groups mentioned in \cite{chetard} (including all abelian groups) are saturated, as we show below. In general, it is much more difficult to compute saturated rings than usual graded character rings, due to the complexity of the saturated filtration. This is where \Cref{swanslemma} comes into play, as we show with the example of the projective special linear group $ PSL(2,q)$.
For convenience, when the groups $	H\leq G$ are clear from the context, we denote the induction $\ind_H^G: R(H) \to R(G)$ by $	i_*$ and the restriction $	\res^G_H: R(G) \to R(H)$ by $	i^*$. All examples are computed over $\CC$.

\subsection{Saturated groups}
Abelian groups, dihedral groups of order $2p$ and the quaternion group of order $	8$ are all saturated.
\begin{Proposition}\label{abeliangroupsaresaturated}
	Abelian groups are saturated.
\end{Proposition}
\begin{proof}
	Let $G$ be an abelian group and define $\widehat{G} := \Hom(G,\CC^*)$. Then any abelian group homomorphism $\phi: G \rightarrow H$ induces a map $\widehat{\phi} : \widehat{H} \rightarrow \widehat{G}$, which is injective if and only if $\phi$ is surjective. Additionally, there is a natural isomorphism between $G$ and its double dual $\widehat{\widehat{G}}$ given by associating to $g$ the evaluation at $g$.\\
	Now if $H\leq G$, then the injection $H \rightarrow G$ induces a map $\widehat{\phi}: \widehat{G} \rightarrow \widehat{H}$, and also a map $\widehat{\widehat{\phi}}:\widehat{\widehat{H}} \rightarrow \widehat{\widehat{G}}$. The latter is injective, which means by the above that $\widehat{\phi}$ is surjective. Thus the characters of $H$ all come from restrictions of characters of $G$, and $G$ is saturated.
\end{proof}

We now turn to the quaternion group 	$ Q_8 =  = \langle i,j,k \ | \ i^2 = j^2 = k^2 = ijk \rangle	$. The group $Q_8$ has 5 conjugacy classes: $\lbrace 1\rbrace$, $\lbrace -1\rbrace$, $\lbrace \pm i\rbrace$, $\lbrace \pm j\rbrace$, $\lbrace \pm k\rbrace$ so 5 irreducible representations on $\CC$. They are as follows:
	\begin{enumerate}
		\item In dimension 1, the trivial representation, and the characters $\rho_1: \begin{cases} i \mapsto 1 \\ j\mapsto -1 \end{cases}$, $\rho_2 = -\rho_1$ and $\rho_3 = \rho_1\rho_2$,
		\item and in dimension 2, the representation $\Delta$:
		\begin{equation*}
		 \Delta(i) = \begin{pmatrix} i & 0 \\ 0 & -i \end{pmatrix}, \ \ \ \Delta(j) = \begin{pmatrix} 0 & -1 \\ 1 & 0 \end{pmatrix}, \ \ \ \Delta(k) = \begin{pmatrix} 0 & -i \\ -i & 0 \end{pmatrix} \begin{matrix} \vphantom{1} \\ \vphantom{0}. \end{matrix}
		\end{equation*}
	\end{enumerate}
We recall the following result from \cite{chetard}:
\begin{Lemma}[{\cite[Th 6.4]{chetard}}] Let $\rho_1$ be the character of $	Q_8$ defined by $	\rho_1(i) = 1$, $	\rho_1(j) = -1$, let $\rho_2 = - \rho_1$, and let $	\Delta$ be irreducible character of degree $2$ of $Q_8$ sending $i,j,k$ to $0$. Then
\[
			R^*(Q_8) = \frac{\ZZ[x_1,x_2,y]}{(2x_i, 8y, x_i^2, x_1x_2 - 4y)},
\]
where $	x = c_1(\rho_1)$, $	y = c_1(\rho_2)$ and $	y = c_2(\Delta)$.
\end{Lemma}

We also need a result from \cite{guillot-minac}:
\begin{Lemma}[{\cite[Prop 3.4]{guillot-minac}}]
	Let $C_N$ be the cyclic group of order $N$ and $\rho$ a generating representation for $R(C_N)$. Then
	\[
		R^*(C_N) = \frac{\ZZ[t]}{(Nt)},
	\]
	where $t = c_1(\rho)$.
\end{Lemma}

\begin{Proposition} \label{Q8_is_saturated}
	The quaternion group $Q_8$ is saturated.
\end{Proposition}
\begin{proof}
	The quaternion group contains one subgroup isomorphic to $C_2$, which is generated by $-1$, and three subgroups isomorphic to $C_4$, which all contain $-1$ and are generated respectively by $i$, $j$ and $k$. Since all these groups are saturated, we only need to check that the maximal saturated subgroup $H = \langle k \rangle \cong C_4$ is $\Gamma$-compatible with $Q_8$, which we do by showing that, if $	\rho$ is the generating representation of $R(C_4)$, then each induced character $\ind_{C_4}^{Q_8}(C_1(\rho)^n)$ is in $\Gamma^n(Q_8)$.
	Note first that $\ind_{C_4}^{Q_8}(C_1(\rho)) \in \Gamma^1(Q_8) = I_{Q_8}$. Moreover, the representation $\Delta$ restricts on $C_4$ to $\rho + \rho^{-1}$, and so
	\[
		\res^{Q_8}_{C_4}(y) = c_2(\rho+\rho^{^-1}) = -c_1(\rho)^2 = -t^2.
	\]
	Therefore $C_1(\rho)^2 = \res(-C_2(\Delta))$, and so
	\[
		i_*(C_1(\rho)^2) = i_*(i^*(-C_2(\Delta))) = - \CC[Q_8/C_4] \otimes C_2(\Delta) \in \Gamma^2(Q_8).
	\]
	Thus, for any $n = 2m+l$ with $l = 0,1$:
	\[
		i_*(C_1(\rho)^n) = i_*(C_1(\rho)^{2m+l}) = i_*\left(C_1(\rho)) \cdot i^*(-C_2(\Delta))^m\right) = i_*(C_1(\rho))\cdot(-C_2(\Delta))^m,
	\]
	which is an element of  $\Gamma^l\cdot\Gamma^{2m}$. This means that $C_4$ is $\Gamma$-compatible with $Q_8$, and therefore $Q_8$ is saturated.
\end{proof}

With a similar method, we can prove that dihedral groups are saturated.
\begin{Lemma}[{\cite[Prop 4.4]{chetard}}]
	Let $p$ be an odd prime, and let $D_p = \langle \sigma, \tau \ | \ \tau^2 = \sigma^p = 1, \tau\sigma\tau = \sigma^{-1} \rangle$ be the dihedral group of order $p$. Let $\chi$ be the irreducible character of $D_p$ of degree 2, sending $\tau$ to $0$ and $\sigma$ to $2cos(\frac{2\pi}{p})$, then
	\[
		R^*(D_p) = \frac{\ZZ\left[x,y\right]}{(2x,py,xy)},
	\]
	where $x = c_1(\chi)$ and $y = c_2(\chi)$.
\end{Lemma}
\begin{Proposition}\label{Dp_is_saturated}
	Let $p$ be an odd prime, then the dihedral group $D_p$ of order $2p$ is saturated.
\end{Proposition}
\begin{proof}
	Since $D_p = C_p \rtimes C_2$ and $C_p$, $C_2$ are abelian, these are the maximal saturated subgroups of $D_p$. The signature $\varepsilon$ of $D_p$ restricts on $C_2$ to the representation $\rho$, which generates $R(C_2)$. Thus $C_2$ is $\Gamma$-compatible with $D_p$, and we only need to look at $C_p$.\\
	Since $\res(Y) = -C_1(\rho)^2$ the same argument as in the proof of \cref{Q8_is_saturated} applies. \\
\end{proof}

		\subsection{Projective linear groups}
We compute the saturated character ring of $G=PSL(2,p)$, the projective special linear group over $\FF_p$, where $p$ is an odd prime such that $p \equiv 3,5 (\Mod 8)$. Note that we do not use any information about the character table of $G$: we only need to know those of its Sylow subgroups, which are all abelian. For each prime $l$ dividing $|G| = \frac{p(p+1)(p-1)}{2}$, let $H_l = Syl_l(G)$ and $N_l = N_G(H_l)$. For each $l$, we determine the $l$-Sylow of $G$ and the action of its normalizer, then deduce the stable element subring. There are $4$ possible cases:
	\begin{enumerate}
		\item $l = p$. Then $H_p \cong C_p$ is generated by the matrix $\begin{pmatrix} 1 & 1 \\ 0 & 1 \end{pmatrix}$. The normalizer of $H_p$	is the group:
		\[
			N_p = \left\lbrace \begin{pmatrix} a & b \\ 0 & a^{-1} \end{pmatrix}  \in PSL(2,p) \right\rbrace,
		\]
		with action
		\[
			\begin{pmatrix} a & b \\ 0 & a^{-1}\end{pmatrix} \cdot \begin{pmatrix} 1 & n \\ 0 & 1 \end{pmatrix} \cdot \begin{pmatrix} a^{-1} & -b \\ 0 & a \end{pmatrix} = \begin{pmatrix} 1 & a^2n \\ 0 & 1 \end{pmatrix},
		\]
		inducing $\rho \mapsto \rho^{a^2}$ for a generator $\rho$ of $R(C_p)$. On $ \rr^*(C_p) \cong \frac{\ZZ[x]}{(px)} $, this induces $x\mapsto a^2x$. The subring generated by $x^{\frac{p-1}{2}}$ is stable by this action, and conversely if $a$ is an element of multiplicative order $(p-1)$, then a monomial $x^m$ being stable by the action $x \mapsto a^2x$ implies that $m$ is a multiple of $\frac{p-1}{2}$. Thus
		\begin{equation}
			\rr^*(H_p)^{N_p} \cong \frac{\ZZ[u]}{(pu)}, \ \ \ \ |u| = \frac{p-1}{2}.
		\end{equation}

		\item $l$ is an odd prime dividing $(p-1)$. Then $H_l \cong C_{l^i}$ for some integer $i$, generated by $\begin{pmatrix} n & 0 \\ 0 & n^{-1} \end{pmatrix} $ for some $n$ of order $l^i$ in $\FF_p\x$. A straightforward computation gives that $N_l$ is generated by diagonal matrices (which commute with the elements of $H_l$) together with the matrix $\begin{pmatrix} 0 & 1 \\ -1 & 0 \end{pmatrix}$ which sends an element $h \in H_l$ to its inverse. The induced action on the representation ring is $\rho \mapsto \rho^{-1}$, which translates as $x \mapsto -x$ in the graded ring. Thus
		\begin{equation}
		\rr^*(H_l)^{N_l} \cong \frac{\ZZ[x]}{(l^ix)}, \ \ \ \ |x| = 2.
		\end{equation}

		\item $l=r$ is an odd prime dividing $p+1$. We prove that $H_r$ is cyclic. Note that the $r$-Sylow of $G$ is isomorphic to that of $G' := PSL(2,p^2)$ since the index of $G$ in $G'$ is coprime to $r$. Let $\alpha \in \FF_{p^2}\x$ have multiplicative order $r^i$.  The matrix $A' = \begin{pmatrix} \alpha & 0 \\ 0 & \alpha^{-1} \end{pmatrix}$ generates a cyclic group $H_r'$ of order $r^i$ in $G'$, which is thus an $r$-Sylow subgroup.
		We have $\alpha \notin \FF_p\x$, however any matrix of $G$ similar to $A$ generates an isomorphic group in $G$. One can take for example $A = \begin{pmatrix} 0 & -1 \\ 1 & \alpha+\alpha^{-1} \end{pmatrix}$, the companion matrix to the minimal polynomial of $\alpha$.

		The normalizer $N_r'$ of $C_{r^i}'$ in $G'$ is a dihedral group of order $p^2-1$, generated by all diagonal matrices together with the matrix $\begin{pmatrix} 0 & 1 \\ -1 & 0 \end{pmatrix}$ which sends $A$ to its inverse. The change of basis sending $A$ to $A'$ allows us to view $N_r$ as a subgroup of $N_r'$, and thus the elements of $N_r$ act either trivially or by inversion on $H_r$.

		It remains to show that there exists a matrix $S \in G$ such that $S^{-1}AS = A^{-1}$. Let $a = \alpha+\alpha^{-1}$. By a direct calculation, one shows that any matrix of the form $\begin{pmatrix} -x & y \\ ax+y & x \end{pmatrix}$ in $GL(2,p)$ satisfies this property, thus $S \in PSL(2,p)$ exists if and only if there is a pair $(x,y) \in \Fp^2$ such that $-x^2 -axy - y^2 = 1$.
		This equation is equivalent to $X^2 +1 = bY^2$, with~$X = x+\frac{a^2}{4}y$, $Y = y$ and~$b =\frac{a^2}{4} -1$. There are $(p+1)/2$ squares in $\Fp$ (including $0$), so there are $(p+1)/2$ elements of the form $X^2+1$, and, if $b \neq 0$ then there are also $(p+1)/2$ elements of the form $bY^2$.
		Thus whenever $b \neq 0$, the set of elements of the form $X^2+1$ and the set of elements of the form $bY^2$ have nontrivial intersection, and there is a solution to $x^2 + axy + y^2 = -1$. Now, $b = 0$ if and only if $a^2 = 4$, that is, $a = \pm 2 (\Mod p)$. But then $\alpha$ is a solution of $t^2 \pm 2t +1$, that is, $\alpha = \alpha^{-1}$ has multiplicative order $2$, in contradiction with our assumption. Thus $b$ is always nonzero, which completes the proof.\\
		We have:
		\begin{equation}
			\rr^*(H_r)^{N_r} \cong \frac{\ZZ[y]}{(r^iy)}, \ \ \ \ |y| = 2.
		\end{equation}

		\item $l = 2$. Since $p \equiv 3,5 (\Mod 8)$, the $2$-Sylow subgroup of $G$ has order $4$. There are two cases:
		\begin{itemize}
			\item if $p \equiv 5 (\Mod 8)$, then $-1$ is a quadratic residue, so let $a$ satisfy $a^2 \equiv -1 (\Mod p)$. Then
			\[
				H_2 = \left\langle h_1 := \begin{pmatrix} a & 0 \\ 0 & -a \end{pmatrix}, h_2 := \begin{pmatrix} 0 & a \\ a & 0 \end{pmatrix} \right\rangle.
			\]
			We show that $N_2 \cong A_4$. First, we have $C_G(h_1)\cap N_G(H_2) = \{Id\}$, as a direct calculation shows, and similarly for $h_2$ and $h_1h_2 =: h_3$. Therefore, if $N \in N_2$ acts nontrivially on $H_2$, it must permute all $3$ nontrivial elements. If $T = \begin{pmatrix} x & -ax \\ x & ax \end{pmatrix}$, with $x^2 = \frac{1}{2a}$, then $Th_1T^{-1} = h_2$ and $Th_2T^{-1} = h_3$. Both $2$ and $a$ are nonresidues $\Mod p$ since $p \equiv 5 (\Mod 8)$ and if $a$ were a residue, then $PSL(2,p)$ would contain an element of order 4, contradicting $H_2 \cong C_2\times C_2$. Thus there is an $x$ satisfying $x^2 = 1/2a$. Moreover $T$ is unique up to multiplication by an element of $C_G(H_2) = H_2$, which shows that $N_2 = \left\langle T, H_2 \right\rangle \cong A_4$.\\
		\item if $p \equiv 3 (\Mod 8)$, then $-2$ is a residue, so let $b$ satisfy $b^2 \equiv -2 (\Mod p)$. Then
			\[
				H_2 = \left\langle \begin{pmatrix} 0 & -1 \\ 1 & 0 \end{pmatrix}, \begin{pmatrix} b & 1 \\ 1 & -b \end{pmatrix} \right\rangle
			\]
			Again, we have $N_2 \cong A_4$ acting by cyclic permutations, generated by $H_2$ together with the matrix $T = \begin{pmatrix} \frac{1}{b} & \frac{1}{b} \\ -\frac{b+2}{2} & \frac{b-2}{2} \end{pmatrix}$.\\
		\end{itemize}
		In both cases the normalizer acts as cyclic permutations on the nontrivial elements of $H_2$, and thus:
			\begin{equation}
				\rr^*(H_2)^{N_2} \cong \frac{\ZZ[z,t]}{(2z,2t, z^3 - t^2)}, \ \ \ \ |z| = 2, |t| = 3.
			\end{equation}
	\end{enumerate}
	Putting all of this together, we get:
	\begin{Theorem} \label{rr*psl2p}
		Let $G = PSL(2,p)$ be the projective special linear group over $\FF_p$, where $p$ is an odd prime such that $p \equiv 3,5 (\Mod 8)$. Write:
		\[
			|G| = 4\cdot p\cdot l_1^{i_1}\cdots l_n^{i_n}\cdot r_1^{j_1}\cdots r_m^{j_m}, \ \ \ \ \text{ with } \ \ l_k|(p-1), \ \  r_k|(p+1).
		\]
		Then:
		\begin{equation}
			\rr^*(G) \cong \frac{\ZZ[x_1,\cdots,x_n, y_1,\cdots y_m, z, t, u]}{(l_k^{i_k}x_k, r_k^{j_k}y_k, 2z,2t, pu, z^3 - t^2)}
		\end{equation}
		with $|x_k| = |y_k| = |z| = 2$, $|t| = 3$, $|u| = (p-1)/2$, and:
		\begin{enumerate}
			\item $x_k = \ind_{H_{l_k}}^G (x_k^2)$, $y_k = \ind_{H_{r_k}}^G(y_k^2)$ where $x_k$ (resp. $y_k$) is a generating class of the ring $R^*_\CC(C_{l_k^{i_k}})$ (resp. $R_\CC(C_{r_k^{j_k}})$).
			\item $u = \ind_{H_p}^G(u)^{(p-1)/2})$ where
			\item $z = \ind_{H_2}^G(t_1^2 + t_1t_2 + t_2^2)$ and $t = \ind_{H_2}^G(t_1^3 +t_1^2t_2 + t_2^3)$
		\end{enumerate}
	\end{Theorem}

\begin{Remark}
	For $p=3$, this is the saturated ring $\rr^*(A_4)$.
\end{Remark}


\section{Tambara functors, the ungraded case} \label{tambara}
	After discussing whether the graded character ring functor is Mackey, it seems natural to turn to the theory of Tambara functors, which was introduced by Tambara in \cite{tambara}; they can be understood as Mackey functors $S(-)$ that are equipped, for each subgroup $H \leq G$, with a multiplicative transfer map $S(H) \to  S(G)$. In cohomology, this is the Evens norm map (see for example \cite[Ch. 6]{carlson-townsley}). In the case of graded character rings, tensor induction of representations is a natural candidate for the role of the multiplicative transfer. We must begin, however, with the ungraded situation: the fact that the multiplicative transfer turns $K_G(X)$ into a Tambara functor is mentioned without proof in both \cite{strickland} and \cite{tambara}, and we propose here a proof for the sake of completeness.

	To define Tambara functors, we need the notion of exponential diagrams. Let $\Gset/X, \Gset/Y$ be the categories of $G$-sets over $X, Y$ respectively, and let an equivariant map $f:X\to Y$ be given. The pullback functor $\Gset/Y \to \Gset/X$ has a right adjoint $\Pi_f: \Gset/X \to \Gset/Y$, which we now describe. Let $p: A\to X$ be a set over $X$.
	We construct $q:\Pi_f A \to Y$ as follows:
	\[
		\Pi_f A = \bigsqcup_{y \in Y} \textrm{sec}_p(f^{-1}(y), A),
	\]
	where we write $\textrm{sec}_p(U,A)$, given a subset $U\subset X$, for the set of all sections of $p$ over $U$, that is, maps $s:U\to A$ such that $p \circ s(u) = u$ for all $u\in U$.

	Then $\Pi_fA$ is a $G$-set if we define $^gs: f^{-1}(gy) \to A$, $x \mapsto g\cdot s(g^{-1}\cdot x)$, and of course there is an obvious map $\Pi_fA \to Y$. The adjointness property means that, as is easily established,
	\[
		\Hom_{\Gset/X}(P(B),A) \cong \Hom_{\Gset/Y}(B, \Pi_fA)
	\]
	for all appropriate $A,B$, where $P$ is the pullback functor. In particular for $B = \Pi_fA$, there is an element $e\in \Hom_{\Gset/X}(P(\Pi_fA), A)$ corresponding to the identity of $B$. It is involved in the following commutative diagram:
	\[
		\xymatrix{ X\times_Y \Pi_fA \ar[d]_{f'} \ar[r]^-e & A \ar[r]^p & X \ar[d]^f \\
			\Pi_fA \ar[rr]_q & & Y \\
		}.
	\]
	Seeing the pullback as pairs $(x, s)$ with $x\in X$, and $s \in \textrm{sec}_p(f^{-1}f(x),A)$ such that $f(x) = q(s)$, the map $f'$ is just a projection on the second coordinate.
	A diagram isomorphic to the one above is called an exponential diagram.
	\begin{Definition*}[{\cite[\S 2]{tambara}}]
		Let	$X,Y$ be $G$-sets and $f: X\to Y$ a $G$-set map. A semi-Tambara functor is a function $S(-)$ associating to $X,Y$, (semi)-rings $S(X), S(Y)$ and to $f$ three maps $f^*: S(Y) \to S(X)$, $f_*, f_\sharp: S(X) \to S(Y)$ such that the following conditions are satisfied:
		\begin{enumerate}
			\item $f^*$, $f_*$, $f_\sharp$ are homomorphism of rings, additive monoids, multiplicative monoids respectively.
			\item The triples $(S, f^*, f_*)$ and $(S, f^*, f_\sharp)$ form semi-Mackey functors.
			\item If
			\[
				\xymatrix{ X' \ar[d]_{f'} \ar[r]^e & Z \ar[r]^p & X \ar[d]^f \\
					Y' \ar[rr]_q & & Y \\
				}
			\]
			is an exponential diagram, then the corresponding diagram
			\[
				\xymatrix{ S(X') \ar[d]_{f'_\sharp}  & S(Z) \ar[l]_-{e^*} \ar[r]^{p_*} & S(X) \ar[d]^{f_\sharp} \\
					S(Y') \ar[rr]_{q_*} & & S(Y) \\
				}
			\]
			commutes.
		\end{enumerate}
		If additionally, $S$ associates a ring to a $G$-set, and $f_*$ is a homomorphism of additive groups, then $S$ is a Tambara functor.
	\end{Definition*}

	\begin{Theorem}
		The functor $K_G^+(-)$ with the restriction, induction and tensor inducion maps described in \Cref{definitions}, is a semi-Tambara functor.
	\end{Theorem}
	\begin{proof}
		Since $(K_G^+(-), f^*, f_*)$ is a Mackey functor, so we only concern ourselves with the properties of $f_\sharp$.
		\begin{enumerate}
			\item The fact that $f_\sharp$ is a homomorphism of multiplicative monoids follows from the properties of the tensor product.
			\item To show that $(S,f^*,f_\sharp)$ is a Mackey functor, we check both axioms from the definition in \Cref{stable_elements}. Let
			\[ \xymatrix{ \Omega_1 \ar[r]^\alpha \ar[d]_\beta & \Omega_2 \ar[d]^\gamma \\
					\Omega_3 \ar[r]_\delta & \Omega_4 }
			\]
			be a pullback diagram of $G$-sets, then we want to check that $\delta^* \gamma_\sharp = \beta_\sharp \alpha^*$. Note that because any $G$-set can be expressed as a disjoint union of orbits, it is sufficient to check this axiom on pullback diagrams of the form:
				\[
				\xymatrix{ \Omega \ar[r]^\alpha \ar[d]_\beta & G/K \ar[d]^\gamma \\
					G/H \ar[r]_\delta & G/J}
				\]
				for $H,K \leq J \leq G$. Let $W$ be a vector bundle over $G/K$ and let $V = \delta^*\gamma_\sharp(W)$, then for $x\in G/H$, we have:
				\[
					V_x = \left(\delta^* \gamma_\sharp (W)\right)_{x} = \bigotimes_{t\in G/K} W_t,
				\]
				where the tensor product is taken over all of the $tK \in G/K$ such that $\gamma(tK) = \delta(xH)$. The action of $g$ takes $V_{x}$ to $V_{g\cdot x}$ and can be written $g\cdot\bigotimes_{tK\in G/K} v_{tK} = \bigotimes_{tK\in G/K}g\cdot v_{g^{-1}tK}$. The vector bundle $E := \beta_\sharp \alpha^*(W)$ is defined by
				\[
					E_x = \left(\beta_\sharp \alpha^*(W)\right)_x = \bigotimes_{s\in G/K} W_s,
				\]
				where the tensor product is taken over all $s$ such that $(sK,xH) \in \Omega$, which is equivalent to requiring $\gamma(tK) = \delta(xH)$. The action of $G$ is given by $g\cdot\bigotimes_{s\in G/K}e_s = \bigotimes_{s\in G/K}g\cdot e_{g^{-1}\cdot s}$, and thus $E$ and $V$ are isomorphic vector bundles.

				For the second axiom, consider two finite $G$-sets $\Omega, \Psi$ and the corresponding inclusion maps $i_\Omega, i_\Psi: \Omega, \Psi \to \Omega\sqcup\Psi$. Then the map $f_\sharp: K_G^+(\Omega)\oplus K_G^+(\Psi) \to K_G^+(\Omega\sqcup\Psi)$ whose components are given by $i_{\Omega,\sharp}, i_{\Psi,\sharp}$ should be an isomorphism. This is obviously the case.

				\item Again, we can assume without loss of generality that $A = G/K, X = G/H, Y = G/J$ where $K \leq H \leq J$. We denote by $\pi_K^H: G/K \to G/H$ the map sending a coset $tK$ to the corresponding coset $tH$ in $G/H$, and similarly for $\pi_K^J$, $\pi_H^J$. Then the set $\Pi_{\pi_H^J}G/K$ above $G/J$ is the set of sections $s : {\pi_{H}^J}^{-1}(yJ) \to G/K$ such that for any $tK \in G/K$ satisfying $tJ = yJ$, we have $s(tH) = tK$. Consider the diagram:
				\[
					\xymatrix{ \Pi_{\pi_H^J}G/K\bigtimes_{G/J} G/K \ar[d]_{f} \ar[r]^-e & G/K \ar[r]^{\pi_K^H} & G/H \ar[d]^{\pi_H^J} \\
						\Pi_{\pi_H^J}G/K \ar[rr]_q & & G/J \\
					}
				\]
				where $e,f$ are projections and $q$ is the map sending a section $s: {\pi_{H}^J}^{-1}(yJ) \to G/K$ to $yJ$. The third axiom for Tambara functors says that the corresponding diagram:
				\[
					\xymatrix{ K_G^+\left(\Pi_{\pi_H^J}G/K\bigtimes_{G/J} G/K\right) \ar[d]_{f_\sharp} & \ar[l]_-{e^*} K_G^+\left(G/K\right) \ar[r]^{{\pi_K^H}_*} & K_G^+\left(G/H\right) \ar[d]^{{\pi_H^J}_\sharp} \\
						K_G^+\left(\Pi_{\pi_H^J}G/K\right) \ar[rr]_{q_*} & & K_G^+\left(G/J\right) \\
					}
				\]
				should commute. For convenience, let $X = \Pi_{\pi_H^J}G/K\bigtimes_{G/J} G/K$. Consider $W \in K_G^+\left(G/K\right)$, then on the one hand:
				\[
					V_{yJ} := \left({\pi_H^J}_\sharp{\pi_K^H}_*(W)\right)_{yJ} = \bigotimes_{xH \subseteq yJ}\left(\bigoplus_{tK \subseteq xH}W_{tK}\right)
				\]
				and on the other hand:
				\[
					E_{yJ} := \left(q_* f_\sharp e^*(W)\right)_{yJ} = \bigoplus_{s\in q^{-1}(y)}\left(\bigotimes_{(s,tK) \in X}W_{tK}\right).
				\]
				The fact that $V_{yJ} \cong E_{yJ}$ as vector spaces comes from the distributivity property of the tensor product with respect to the direct sum, as well as the definition of the exponential functor $\Pi_{\pi_H^J}: \Gset/(G/H) \to \Gset/(G/J)$. To construct each term of the sum in $E_{yJ}$, we pick a section $s: {\pi_H^J}^{-1}(yJ) \to G/K$. Each term is then a product of all spaces of the form $W_{tK}$ with $tK = s(xH)$ for $xH \in {\pi_H^J}^{-1}(yJ)$. Summing over all possible such sections $s$, we get all possible combinations of factors in $V_{yJ}$.
				So $E_{yJ}$ is just a rewriting of $V_{yJ}$. The action of $g\in G$ is given by
				\[
				g\cdot \left(\bigotimes_{xH\subseteq yJ}\left(\bigoplus_{tK\subseteq xH} w_{tK}\right)\right) \mapsto \bigotimes_{xH\subseteq yJ}\left(\bigoplus_{tK\subseteq xH} g\cdot w_{g^{-1}\cdot tK}\right).
				\]
				On the other hand, the action of $G$ on $\Pi_{\pi_H^J}G/K$ is given by $g\cdot s = {\pi_H^J}^{-1}(gyJ) \to G/K, \ gxH \mapsto g\cdot s(g^{-1}(gxH))$, that is, $g\cdot s$ maps $g\cdot xH$ to $g\cdot s(xH)$. This means that the permutation of the factors induced by the action of $g$ on $E$ is the same as the one on $V$.
			\end{enumerate}
	\end{proof}

 	The following result by Tambara shows that, in fact, $K_G(X)$ is a Tambara functor. For an abelian monoid $M$, let $\gamma M$ be the universal abelian group with monoid map $k_M: M \to \gamma M$, and generators $k_M(m)$ for $m\in M$ and relations $k_M(m+m') = k_M(m)+k_M(m')$ for $m,m' \in M$. If $M$ is a semi-ring, then $\gamma M$ has a unique ring structure such that $k_M$ is a semi-ring map.
	\begin{Theorem}[{\cite[Th. 6.1]{tambara}}]
		Let $S$ be a semi-Tambara functor. Then the function which assigns the set $\gamma S(X)$ to each $G$-set $X$ has a unique structure of a Tambara functor such that the maps $k_{S(X)}$ form a morphism of semi-Tambara functors.
	\end{Theorem}

	\begin{Corollary}\label{kgtambara}
		The functor $K_G(-)$ has the structure of a Tambara functor.
	\end{Corollary}


\section{The addition formula} \label{additionformula}
 A formula for the norm of the sum of two characters would enable us to compute the value of the norm map on negative virtual characters, a necessary step in determining whether the norm map preserves the Grothendieck filtration. Strikingly, there is no known general formula for the tensor induction of a sum of characters, or its cohomological equivalent, the Evens norm of a sum of classes. Below, we first establish a formula for the sum of two positive characters after \cite[\S 4]{tambara}; we then use this formula to determine $\nn_H^G(-\rho)$ for $\rho \in R^+(H)$, in the case of a normal subgroup $H$ of prime index in $G$, which gives us an explicit expression for the norm of a virtual character in this case. We then prove that in the case of abelian groups, the norm map preserves the Grothendieck filtration, and thus $R^*(G)$ is a Tambara functor on abelian groups.

	\subsection{A general formula for positive representations}
The following is an application of \cite[\S 4]{tambara}, where Tambara gives a general addition formula for the norm. Let $X,Y$ be $G$-sets and let $f: X\to Y$ be a $G$-map. As usual, we assume $X = G/H$, $Y = G/K$ with $H\leq K \leq G$, and $f = \pi_K^H$. Moreover, we can restrict ourselves to $K = G$. So $Y = G/G = \bullet$, the one point set. Let:
\begin{align*}
	V &= \{ C \ | \ C \subset G/H \} =: \mathcal{P}(G/H) \\
	U &= \{ (xH,C) \ | \ xH\in C, C \subset G/H \}
\end{align*}
Then we have a commutative diagram
\[
		\xymatrix{ U \ar[r]^{r} \ar[d]_{t} & G/H \ar[d]^{f} \\
							V \ar[r]_{s} & \bullet \\}
\]
where $r, t$ are projection maps. Let
\[
  \chi := t_\sharp r^*: K_G^+(G/H) \to K_G^+(V),
\]
then for a vector bundle $W \in K_G^+(G/H)$, the vector bundle $\chi(W) \in K_G^+(V)$ associates to each $C = \{x_1,\cdots,x_n\}$ the vector space $W_{x_1}\otimes\cdots\otimes W_{x_n} \cong W^{\otimes n}$, and to each $g\in G$ the linear map given by:
\[
  g\cdot\left(\bigotimes_{x_i\in C}w_{x_i}\right) = \bigotimes_{x_i \in C}\left(g\cdot w_{g^{-1}\cdot x_i}\right).
\]
So the reader must bear in mind that, in our current notation, for a representation $\rho$ of $H$ corresponding to some vector bundle $x\in K_G^+(G/H)$, we have:
\begin{equation} \label{normischi}
  f_\sharp(x) = \nn_H^G(x) = \chi(x)_{G/H}
\end{equation}
Note that, since it involves the map $t_\sharp$, the morphism $\chi$ is only defined on $K_G^+(G/H)$ for now. Throughout this section, we determine how to extend $\chi$ to vector bundles with (all) negative coefficients, then to all virtual bundles.

We define a ring operation $\vee$ on the group $K_G(V)$ as follows: let $V^{(2)}$ be the $G$-set of pairs $(C_1,C_2)$ of disjoint subsets of $G/H$. Let $p_1,p_2,m$ be the $G$-maps taking $(C_1,C_2)$ to $C_1$, $C_2$, $C_1\sqcup C_2$, respectively. Then for $z, t \in K_G(V)$, we let:
\[
	(z\vee t) = m_*(p_1^*(z)\cdot p_2^*(t)).
\]
This operation does not involve multiplicative norms (that is, it does not involve $f_\sharp$ for some map $f$), thus it is well-defined on the whole ring $K_G(V)$, and not just the semi-ring $K_G^+(V)$.

Each fiber in a vector bundle is a representation of the stabilizer of the point above which it sits; for purposes of intuition, we point out that, as a representation of $\st (C)$, we have
\begin{equation} \label{veeisind}
	 (z\vee t)_C = \bigoplus \ind_{\st(C_1, C_2)}^{\st(C)}(z_{C_1}\otimes t_{C_2}),
\end{equation}
where the direct sum is taken over all orbit representatives under $\st(C)$ of pairs $(C_1,C_2)$ such that $C_1\sqcup C_2 = C$. Since this operation only involves restrictions and inductions, it is defined for virtual characters.

By \cite[Prop. 4.4]{tambara}, the map $\chi$ is is a morphism from the monoid $(K_G^+(G/H), +)$ to $(K_G(V), \vee)$. In particular, for $\tau,\sigma\in K_G^+(G/H)$, we have:
\begin{equation}\label{sumisvee}
		f_\sharp (\sigma + \tau) = \chi(\sigma+\tau)_{G/H} =  \left(\chi\left(\sigma\right)\vee\chi\left(\tau\right)\right)_{\{G/H\}}.
\end{equation}

We now assume that $H$ is a normal subgroup of $G$. In terms of representations, we introduce the following notation for purposes of intuition: for a representation $\rho \in R^+(H)$ and $C\subset G/H$, write
\begin{equation}\label{otimespower}
	\rho^{\otimes C} := \chi(\rho)_C.
\end{equation}
This is meant to remind us of the following description. Pick a transversal set $T = \{t_1,\cdots t_n\}$ for $G/H$ and let $C\subset G/H$. Then, as a $\st C$-module, we have:
\[
  \rho^{\otimes C} = \bigotimes_{t_i}\rho^{t_i},
\]
where $\rho^{t_i}$ is the representation $\rho$ conjugated by $t_i \in G$, and the sum is over those $t_i$, whose image in $G/H$ is in $C$; the action of $\st C$ is obvious. Note that, because $H$ is normal in $G$, the representation $\rho^{\otimes C}$ does not depend on the choice of transversal set $T$, since a different coset representative $t_i'$ of $t_iH$ would be $t_i^h$ for some $h\in H$, and $\rho$ is invariant under conjugation by an element of $H$.

As we recall below, one can extend $\chi$ to virtual characters. However, we shall refrain from using the notation $\rho^{\otimes C}$ when $\rho$ is not known to be an actual representation, as it can be misleading.
For example, if $\rho \in R^+(H)$, and one writes $(-\rho)^{\otimes C}$ for $\chi(-\rho)_C$, then one is tempted to guess that $\chi(-\rho)_C = \pm \chi(\rho)_C$; while \Cref{normprimenormal} establishes just that when $H$ has odd, prime index in $G$, \Cref{normindex2} shows that it is erroneous in general.

Putting \cref{normischi,veeisind,sumisvee,otimespower} together yields:
\begin{Proposition}\label{additionformula1}
	Let $\sigma, \tau \in R^+(H)$, where $H\trianglelefteq G$. Then:
	\[
		\nn_{H}^{G}(\tau + \sigma) = \sum_{C\in \mathcal{O}(V)}\ind_{\mathrm{Stab} C}^{G}\left(\tau^{\otimes C}\sigma^{\otimes C'}\right)
	\]
	where $C'$ denotes the complement of $C$ in $G/H$, and $\mathcal{O}(V)$ is a complete set of orbit representatives of $V$ under the action of $G$.
\end{Proposition}
Again, this formula does not depend on the choice of orbit representatives, since choosing different representatives boils down to conjugating $\tau^{\otimes C}\sigma^{\otimes C'}$ by some $g\in G$, under which induction of representations is invariant. Note that, so far, the formula in \Cref{additionformula1} is only valid on $R^+(H)$.

Let us point out that, if $G$ is abelian, then $\rho^{\otimes C} = \rho^{\otimes|C|}$ as an $H$-module; there results a simplified formula for $\nn_H^G(\sigma+\tau)$ in this case, especially when $H$ has prime index in $G$. What we establish in the sequel is that this simplified formula holds \textit{even when $\sigma$ and $\tau$ are virtual}. This will be \Cref{abelianadditionvirtual}.

A key argument of the proof of \cite[Th 6.1]{tambara}, is that the image of $\chi$ lies in a subset of $K_G(V)$ that is a group for $\vee$. Thus, defining $\chi(-\tau)$ as the element $b \in K_G(V)$ such that $\chi(\tau)\vee b = 1$, extends $\chi$ to virtual characters in a way compatible with the addition formula.
With $b$ thus defined, one has $\nn_{H}^{G}(-\tau) = b_{\{G/H\}}$, and the equation $\nn_{H}^{G}(\sigma - \tau) = \chi(\sigma)\vee\chi(-\tau)$ is an explicit formula for the norm of any virtual character. Unfortunately, this formula is quite delicate to apply in practice, as the example below shows.

	\subsection{The prime normal case} We first extend the norm map to negative bundles, which allows us to determine an addition formula in the case where $H$ is a normal subgroup of $G$ of prime index $p$. We denote (slightly abusively) the extention of $\chi$ to negative (and generally, all virtual) characters, by $\chi$ as well. Throughout, we use the fact that a vector bundle above a $G$-set $X$ is entirely determined by its fibre above each point $x$ and the action of $\st x$ on it; any equality of vector spaces above a point $x$ is to be understood as a canonical isomorphism of $\st x$-modules.
We start with the case where $p$ is odd:
\begin{Proposition}\label{normprimenormal}
	Let $H\trianglelefteq G$ with $|G:H| = p$ an odd prime and let $W \in K_G^+(G/H)$. Then for any $C\subseteq G/H$:
	\[
    \chi(-W)_C = (-1)^{|C|}\chi(W)_C.
  \]
  In particular, for $C = G/H$:
  \[
		f_\sharp(-W) = -f_\sharp(W).
	\]
\end{Proposition}
\begin{proof}
Let $a = \chi(W) \in K_G(V)$. Let $b \in K_G(V)$ satisfy $a\vee b = 1$, that is, $(a\vee b)_C = 0$ for any $C\neq \emptyset$. We proceed by induction on the cardinality of $C$.
Note that, because $H$ is normal of prime index in $G$, for any $C \subsetneq G/H$ we have $\st(C) = H$; thus, the vector space $\chi(W)_C := \otimes_{c\in C} W_c$, is an $H$-module (but not a $G$-module). When $C = G/H$, the vector space $\chi(W)_C$ is a $G$-module. We first treat the case $C \neq G/H$.

Because it is normal in $G$, the subgroup $H$ stabilizes each $c\in C$ individually, and thus for any $C_1\sqcup C_2 = C$ we have $a_{C_1}\otimes a_{C_2} \cong a_C$ as $H$-modules. To declutter notation, we write eg. $a_{xy}$ for $a_{\{x,y\}}$.
Let us explicitly state the first few steps of the induction, assuming in each case that $C\subsetneq G/H$. Throughout, we use \Cref{veeisind} repeatedly:
\begin{itemize}
	\item If $C = \emptyset$ then $(a\vee b)_C = 1$.
	\item If $|C| = 1$ and (say) $C = \{x\}$ then:
  \[
    (a\vee b)_C = \ind_{\st x}^H(a_x\otimes 1) + \ind_{\st x}^H(1\otimes b_x) = \ind_H^H(a_x)+\ind_H^H(b_x)
  \] and so $b_x = -a_x$ as $H$-modules.
	\item If $|C| = 2$, say $C= \{x,y\}$, then (omitting tensor product signs for simplicity of notation):
	\begin{align*}
		(a\vee b)_C &= \ind_H^H(a_{xy})+ \ind_H^H(a_xb_y) + \ind_H^H(a_yb_x) + \ind_H^H(b_{xy}) \\
		 	&= b_{xy} + a_{xy} -a_xa_y - a_ya_x,
	\end{align*}
	 Since $a_{xy} \cong a_xa_y$ as $H$-modules, we have
	\[
		b_{xy} = a_xa_y = a_{xy}.
	\]
	\item If $|C| = 3$, say $C= \{x,y,z\}$, then $a_{xyz} = a_{xy}a_z = a_xa_{yz}$ as $H$-modules, and:
	\begin{align*}
		b_{xyz} =& \ind_H^H(-a_{xyz}) + \ind_H^H(a_xa_{yz}) + \ind_H^H(a_ya_{xz}) + \ind_H^H(a_za_{xy}) \\
    &+ \ind_H^H(a_{xy}a_z) + \ind_H^H(-a_{xz}a_y) + \ind_H^H(- a_{yz}a_x)\\
		=&-a_{xyz}.
	\end{align*}
	\item Suppose $b_C = (-1)^{|C|}a_C$ for $|C|<n$ and take $|C| = n<|G:H|$. Then
	\begin{align*}
		(a\vee b)_C &= \sum_{D\subset C}\ind_{\st(C\setminus D, D)}^{\st C}(a_{C\setminus D}b_D) = b_C + \sum_{D\subsetneq C} \ind_H^H((-1)^{|D|}a_{C\setminus D}a_{D}) \\
			&=b_C + \sum_{i = 0}^{n-1}(-1)^i\binom{n}{i}a_C \\
			&= b_C +(-1)^{n+1}a_C,
	\end{align*}
	and thus $b_C = (-1)^na_C$.
  \end{itemize}
  The induction is complete.

  We can now treat the case $C = G/H$. Then $\st(C) = G$, and $a_D\otimes a_{C\setminus D}$ is not a $G$-module for any $D\subsetneq C$. Let $\mathcal{O}(V)$ be a set of orbit representatives for the action of $G$ on $V$, then, as $G$-modules, we have:
	\begin{align*}
		(a\vee b)_C &= a_C + b_C +\sum_{D \in \mathcal{O}(V)}\textrm{Ind}_{H}^{G} \left( (-1)^{|D|}a_D\otimes a_{C\setminus D}\right) \\
			&= a_C + b_C + \textrm{Ind}_H^G\left(\sum_{D\in \mathcal{O}(V)}(-1)^{|D|}a_D\otimes a_{C\setminus D} \right).
	\end{align*}
	One can pair the summands $(-1)^{|D|}(a_D\otimes a_{C\setminus D})$ and $(-1)^{|C\setminus D|}(a_{C\setminus D}\otimes a_D)$, which are isomorphic and of opposite sign (since $p$ is odd). Hence the terms of the sum cancel and we have $b_C = -a_C$, that is
	\[
		f_\sharp(-W) = -f_\sharp(W).
	\]
\end{proof}

Note that the first step of the proof shows:
\begin{Proposition}\label{normindex2}
	Let $H \leq G$ be a subgroup of index $2$ and $W\in K_G^+(G/H)$. Let $t$ be a representative for the non-trivial coset in $G/H$. Then, in $K_G(\bullet)$,
	\[
		f_\sharp(-W) = -f_\sharp(W) + f_*(W\otimes W_t)
	\]
\end{Proposition}
\begin{proof}
	In this case, the orbit representatives of ordered pairs of disjoint sets of cosets of $H$ in $G$ are $(\{1\}, \{t\})$, $(\{1,t\}, \emptyset)$ and $(\emptyset, \{1,t\})$. Using notation as in the above proof of \Cref{normprimenormal}, we have
	\[
			(a\vee b)_{G/H} = \ind_{\st (\{1,t\},\emptyset)}^G(a_{1,t}) + \ind_{\st (\{1\},\{t\})}^G(a_1\otimes b_t) + \ind_{\st (\emptyset,\{1,t\})}^G b_{1,t}.
	\]
	Since $\st(\{1,t\}, \emptyset) = \st(\emptyset, \{1,t\}) = G$ and $\st (\{1\},\{t\}) = H$, this becomes
	\[
		b_{1,t} = -a_{1,t} + \ind_H^G(a_1\otimes a_t).
	\]
\end{proof}
The above yields a formula for differences of characters, as follows:
\begin{Corollary} \label{subtractionformulanormalprime}
	Let $\rho,\sigma \in R^+(H)$ for any finite group $H$, and suppose $H\triangleleft G $ with $|G:H| = p$ prime. Then
	\[
	 \nn_H^G(\sigma-\tau) = \begin{cases} \nn_H^G(\sigma) - \nn_H^G(\tau) + \sum_{C\in    \mathcal{O}(V)}(-1)^{p-|C|} \ind_H^G(\sigma^{\otimes C}\tau^{\otimes C'}) & \text{ if } p \text{ is odd} \\
      \nn_H^G(\sigma) - \nn_H^G(\tau) + \ind_H^G(\tau \otimes \tau^t) - \ind_H^G(\sigma \otimes\tau^t) & \text{ if } p=2, \end{cases}
  \]
  where $t$ is a representative for the non-trivial coset in $G/H$ when $p=2$.
\end{Corollary}
\begin{proof}
	When $p$ is odd, by \Cref{normprimenormal}, we have $\chi(-\tau)_C = (-1)^{|C|}\chi(\tau)_C$ for $C\subseteq G/H$. Thus
	\begin{align*}
		\nn_H^G(\sigma-\tau) &= (\chi(\sigma)\vee\chi(-\tau))_{G/H} \\
			&= \nn_H^G(\sigma) - \nn_H^G(\tau) + \mathrm{Ind}_H^G\left(\sum_{C\in\mathcal{O}(V)}(-1)^{p-|C|}\chi(\sigma)_C\otimes \chi(\tau)_{C'} \right)\\
			&= \nn_H^G(\sigma)-  \nn_H^G(\tau) + \sum_{C\in \mathcal{O}(V)} (-1)^{p-|C|}\mathrm{Ind}_H^G(\sigma^{\otimes C}\tau^{\otimes C'}).
	\end{align*}
  If $p =2$:
  \begin{align*}
    \nn_H^G(\sigma-\tau) &= (\chi(\sigma)\vee\chi(-\tau))_{G/H} \\
      &= \nn_H^G(\sigma) - \nn_H^G(\tau) + \ind_H^G(\tau\otimes\tau^t) -\ind_H^G(\sigma\otimes\tau^t).
  \end{align*}
\end{proof}

  \subsection{The abelian case}\label{normabelian}
The next step in our derivation is to simplify expressions of the type $x^{\otimes C}$, which are \textit{a priori} defined in the case of actual characters but not for virtual ones. In the abelian case however, the group action of $G$ on $H$ is trivial, and the notation $x^{\otimes C}$ can be extended to cover any virtual character $x$.

\begin{Proposition}\label{chivirtualabelian}
  Let $H\leq G$ be abelian groups with $|G:H| = p$ a prime number. Let $x \in R(H)$ be any virtual character, then for any $C \subsetneq G/H$:
  \[
    \chi\left(x\right)_C = \bigotimes_{t \in C}x^t = x^{|C|} \text{ as (virtual) }H\text{-modules}.
  \]
\end{Proposition}
Recall that the morphism $\chi$ was originally only defined for actual characters, then extended to negative characters. \Cref{chivirtualabelian} says that in the abelian prime case, the naive extension of $\chi$ to all virtual characters is the right one.
\begin{proof}
  First note that the statement is trivial when $p=2$. For $p$ odd, write $x = x^+ - x^-$ with $x^+,x^- \in R^+(H)$. Let $C\subset G/H$, and for any $D\subset C$ let $D'$ be the complement of $D$ in $C$. Recall that by \Cref{normprimenormal}, for any $\rho\in R^+(H)$, we have $\chi(-\rho)_C = (-1)^{|C|}\chi(\rho)_C$. Thus:
  \begin{align*}
    \chi(x)_C &= (\chi(x^+)\vee\chi(-x^-))_C \\
      &= \bigoplus_{D \subset C}\chi(x^+)_D \otimes \chi(-x^-)_{D'}\\
      &= \bigoplus_{D \subset C}(x^+)^{|C|}\otimes (-1)^{|D'|}(x^-)^{|D'|}\\
      &= \bigoplus_{i = 0}^{|C|}\binom{p}{i}(x^+)^i(-x^-)^{|C|-i} \\
      &= x^{|C|}
  \end{align*}
  which proves the statement.
\end{proof}
Thus we can extend the addition formula to all virtual characters:
\begin{Lemma} \label{abelianadditionvirtual}
  Let $x, y \in R(H)$ be virtual characters and suppose $|G:H| = p$ is prime. Then:
  \[
    \nn_H^G(x+y) = \nn_H^G(x) + \nn_H^G(y) + \sum_{i = 1}^{p-1}\ind_H^G(x^{i}y^{p-i}).
  \]
\end{Lemma}
\begin{proof}
  Recall that $\nn_H^G(x+y) = (\chi(x+y))_{G/H}$. So:
  \begin{align*}
    \nn_H^G(x+y)&= (\chi(x)\vee\chi(y))_{G/H} \\
      &= \nn_H^G(x) + \nn_H^G(y) + \sum_{C \in \mathcal{O}(V)}\ind_H^G(x^{|C|}y^{|C'|}) \\
      &= \nn_H^G(x) + \nn_H^G(y) + \sum_{i=1}^{p-1}\ind_H^G(x^i y^{p-i}).
  \end{align*}
\end{proof}
Note that this formula is also valid for $p = 2$.
\begin{Lemma}\label{normabeliandim1}
	Let $H\leq G$ be finite abelian groups with $[G:H] = n$, and let $\rho \in R^+(H)$ be a representation of degree 1. If $\overline{\rho} \in R^+(G)$ satisfies $\rho = \res_H^G(\overline{\rho})$, then:
	\[
		\nn_H^G(\rho) = \overline{\rho}^n.
	\]
\end{Lemma}
\begin{proof}
	Recall that $\nn_H^G(\rho) = \rho^{\otimes G/H}$, viewed as a representation of $G$. Since $H,G$ are abelian groups, we have $\rho^g = \rho$ for any $g\in G$, so that $\nn_H^G(\rho)$ is given by $\rho^n$ on $H$. If $g\notin H$, then since $|G:H| = n$ we have $g^n \in H$ and $\nn_H^G(\rho)(g) = \rho(g^n)$. Here we use, crucially, the fact that $\rho$ has dimension $1$.
	Thus $\nn_H^G\rho$ is determined by its values on $H$, and if $\res_H^G(\overline{\rho}) = \rho$ then $\overline{\rho}^n = \nn_H^G(\rho)$.
\end{proof}
Let us now restrict to the case $\KK = \CC$. Then the irreducible characters of $G$ are one-dimensional and we can apply the above result.
\begin{Corollary}\label{abelianadditionformula}
	Let $H\leq G$ be abelian groups and $|G:H| = p$ be a prime. Let $\sigma, \tau \in R^+_\CC(H)$ be one-dimensional representations, and let $\bar{\sigma}$ (resp. $\bar{\tau}$) satisfy $\res_H^G\bar{\sigma} = \sigma$ (resp. $\res_H^G\bar{\tau} = \tau$). Then, if $p$ is odd:
	\begin{align*}
		\nn_H^G(\sigma+\tau) &= \bar{\sigma}^p + \bar{\tau}^p +\CC[G/H] \sum_{i = 1}^{p-1}\frac{1}{p} \binom{p}{i} \bar{\sigma}^i\bar{\tau}^{p-i}, \\
		 \nn_H^G(\sigma-\tau) &= \bar{\sigma}^p - \bar{\tau}^p +\CC[G/H] \sum_{i = 1}^{p-1}\frac{1}{p} \binom{p}{i}(-1)^{p-i} \bar{\sigma}^i\bar{\tau}^{p-i}.
	\end{align*}
  If $p=2$:
  \begin{align*}
		\nn_H^G(\sigma+\tau) &= \bar{\sigma}^2 + \bar{\tau}^2 -\CC[G/H] \bar{\sigma}\bar{\tau}, \\
		 \nn_H^G(\sigma-\tau) &= \bar{\sigma}^2 - \bar{\tau}^2 +\CC[G/H]\bar{\tau}^2 -\CC[G/H]\bar{\sigma}\bar{\tau}.
  \end{align*}
\end{Corollary}
\begin{proof}
	If $G$ is abelian then the action of $G$ on $H$ is trivial and $\sigma^{\otimes C} = \sigma^{|C|}$ as $\st C$-modules for all $C\subset G/H$, trivially when $C$ is proper, and by \Cref{normabeliandim1} otherwise. The number of subsets $C$ of $G/H$ of size $i$ is $\binom{p}{i}$ which we divide by $p$ to sum over orbits.
	Moreover, the induced representation $\ind_H^G(\rho)$ is $\CC[G/H]\bar{\rho}$, for any $\bar{\rho}$ such that $\res_H^G\bar{\rho} = \rho$.
\end{proof}

\subsection{Norm and Grothendieck filtration}
We can now show our main theorem:
\begin{Theorem}\label{normingamma}
	Let $H\leq G$ be abelian with $[G:H] = p$ a prime number. If $x\in \Gamma^n(H)$ then $\nn_H^G(x) \in \Gamma^{np}(G)$.
\end{Theorem}
\begin{proof}
	Recall that $\Gamma^n(H)$ is generated by elements of the form $(\rho_1-1)^{i_1}\cdots(\rho_k-1)^{i_k}$ for $\rho_i$ irreducible characters of $H$ and $\sum i_k \geq n$. Note that, in our case, each $\rho_i$ is one-dimensional.

  As a first step, we prove that if $\rho$ is any irreducible character of $H$, then $\nn_H^G(\rho -1) \in \Gamma^p(G)$. First, assume $p=2$. In the notation of \Cref{abelianadditionformula}, we have:
  \begin{align*}
    \nn_H^G(\rho - 1) &= \bar{\rho}^2 -1 + \CC[G/H] - \CC[G/H]\bar{\rho} \\
      &= (\bar{\rho} - 1)^2 +2(\bar{\rho} - 1) - \CC[G/H](\bar{\rho} - 1) \\
      &= (\bar{\rho} - 1)^2 -(\CC[G/H] -2)(\bar{\rho} - 1) \ \in \Gamma^2(G).
  \end{align*}
  If $p$ is odd:
	\begin{align*}
		\nn_H^G(\rho-1) &= (\bar{\rho})^p - 1 +\CC[G/H] \sum_{i = 1}^p\frac{1}{p} \binom{p}{i}(-1)^{p-i} (\bar{\rho})^i.
	\end{align*}
	Consider the permutation representation $\CC[G/H]$ and recall that $\CC[G/H] = 1+\sigma+\cdots+\sigma^{p-1}$ for some linear representation $\sigma$ of $G$. Let $Y = \sigma -1$, then $Y^i \in \Gamma^i(G)$ and:
	\[
		\CC[G/H] = \sum_{i = 0}^{p-1} (Y+1)^i = \sum_{i = 0}^{p-1}\binom{p}{i+1}Y^i.
	\]
	Note that $(Y+1)^p -1= 0$, and thus
	\[
		pY = -\sum_{i = 2}^{p}\binom{p}{i}Y^i \ \in \Gamma^2(G).
	\]
	Thus we can substitute every instance of $pY$ in the right-hand-side of the equation by $-\sum_{i = 2}^{p}\binom{p}{i}Y^i$. Iterating, we obtain that $pY \in \Gamma^p(G)$, and thus:
	\[
		\CC[G/H] \equiv p \ (\Mod \Gamma^p(G)).
	\]
	Therefore
	\begin{align*}
		\nn_H^G(\rho-1) &= (\bar{\rho})^p - 1 +\CC[G/H] \sum_{i = 1}^p\frac{1}{p} \binom{p}{i}(-1)^{p-i} (\bar{\rho})^i \\
			&\equiv (\bar{\rho})^p - 1 + \sum_{i = 1}^p \binom{p}{i}(-1)^{p-i}(\bar{\rho})^i \ (\Mod \Gamma^p(G)) \\
			&\equiv (\bar{\rho} -1)^p \ (\Mod \Gamma^p(G)) \\
			&\equiv 0 \ (\Mod \Gamma^p(G)),
	\end{align*}
  Which completes the first step.

	Since the norm map is multiplicative, we have, for any prime $p$,
	\[
		\nn_H^G\left((\rho_1-1)^{i_1}\cdots(\rho_k-1)^{i_k}\right) \in \Gamma^{pn}(G)
	\]
	whenever $\sum i_k \geq n$.

  Finally, let $x,y \in R(H)$ be generators of $\Gamma^n(H)$ of the form $(\rho_1-1)^{i_1}\cdots(\rho_k-1)^{i_k}$, as above. Then by \Cref{abelianadditionvirtual}, we have:
  \[
    \nn_H^G(x+y) = \nn_H^G(x) + \nn_H^G(y) + \sum_{i=1}^{p-1}\ind_H^G(x^{i}y^{p-i}).
  \]
  The terms $\nn_H^G(x), \nn_H^G(y)$ are both in $\Gamma^{np}(G)$ as shown above. Each term $x^iy^{p-i}$ is a product of $n$ elements in $\Gamma^n(H)$ and therefore in $\Gamma^{np}(H)$, and since induction preserves the Grothendieck filtration in the case of abelian groups, we have $\ind_H^G(x^iy^{p-i}) \in \Gamma^{np}(G)$ for all $i$; this concludes the proof.
\end{proof}
Thus, on the complex field $\CC$, tensor induction preserves the Grothendieck filtration. Since it satisfies the compatibility axioms of a Tambara functor on $R_\CC(-)$, it satisfies them at the graded level, and we have the following corollary:
\begin{Corollary}\label{abeliantambara}
	The restriction of $R^*_\CC(-)$ to abelian groups is a Tambara functor. \qed
\end{Corollary}

	\section{Application: norms in graded character rings of abelian groups} \label{application}
 In group cohomology Steenrod operations can be defined via the Evens norm corresponding to the inclusion $G \hookrightarrow G\times C_p$ (see eg. \cite[Ch. 7]{carlson-townsley}). We propose here to compute that norm in the case of degree 1 classes in abelian groups. Let $G$ be abelian, and $\sigma$ a one-dimensional representation of $G$, with $x := c_1(\sigma)$.

Recall that $R_\CC(C_p)$ is generated by one character $\rho$ and that
\[
  R_\CC(G\times C_p) \cong R_\CC(G)\otimes R_\CC(C_p).
\]
In $R_\CC(G\times C_p)$, let $\overline{\rho} = 1\otimes \rho$ and $\overline{\sigma} =  \sigma \otimes 1$. In $R^*_\CC(G\times C_p)$, let $y = c_1(\overline{\rho})$ and $z = c_1(\overline{\sigma})$.

\begin{Proposition}
  With notation as above:
  \[
    \nn_G^{G\times C_p}(x) = z^p - zy^{p-1}.
  \]
\end{Proposition}
\begin{proof}
  Let $X = \rho - 1 \in R_\CC(G)$ and $Y = \overline{\rho} - 1, Z = \overline{\sigma} - 1 \in R_\CC(G\times C_p)$. We compute $\nn_G^{G\times C_p}(X)$ in the cases $p = 2$ and $p$ odd.
  \begin{itemize}
    \item Case $p = 2$. By \Cref{abelianadditionformula}:
    \begin{align*}
      \nn_{G}^{G\times C_2}(X) &= \nn_{G}^{G\times C_2}(\rho - 1) = \overline{\rho}^2 - 1^2 + \CC[G\times C_2/G] - \CC[G\times C_2/G]\cdot \overline{\rho} \\
          &= (\overline{\rho} - 1)^2 + 2(\overline{\rho}-1) - \CC[G\times C_2/G](\overline{\rho} -1) \\
          &= (\overline{\rho} -1)^2 - (\overline{\sigma} +1 - 2)(\overline{\rho} -1) \text{ since } \CC[G\times C_2/G] = \overline{\sigma} +1 \\
          &= (\overline{\rho} -1)^2 - (\overline{\sigma} - 1)(\overline{\rho} - 1) \\
          &= Z^2 - ZY.
    \end{align*}
    In the graded ring $R^*_\CC(G\times C_2)$, this yields
    \[
      \nn_{G}^{G\times C_2}(x) = z^2 - zy.
    \]

    \item Case $p$ odd. Let $K = G\times C_p$, then, again by \Cref{abelianadditionformula}:
    \begin{align*}
      \nn_G^K(X) &= \nn_G^K(\rho - 1)\\
        &= \nn_G^K(\rho) + \nn_G^K(-1) + \CC[K/G]\sum_{i=1}^{p-1}\frac{1}{p}\binom{p}{i}\overline{\rho}^i (-1)^{p-i} \\
        &= \overline{\rho}^p - 1 + (\CC[K/G] - p)\left(\sum_{i=1}^{p-1}\frac{1}{p}\binom{p}{i}\overline{\rho}^i(-1)^{p-i}\right) + \sum_{i=1}^{p-1}\binom{p}{i}\overline{\rho}^i(-1)^{p-i}.
    \end{align*}
    Recall that $\overline{\rho} = Z+1$ and that:
    \begin{align*}
      \CC[K/G] &= 1+\overline{\sigma} + \cdots + \overline{\sigma}^{p-1} \\
        &= \sum_{i=0}^{p-1}(Y+1)^i.
    \end{align*}
    thus
    \begin{align*}
        \nn_G^K(X) =& (Z+1)^p - 1 + \sum_{i=1}^{p-1} \binom{p}{i}(Z+1)^i(-1)^{p-i} \\
          &+\left(\sum_{i = 0}^{p-1}(Y+1)^i - p\right)\left(\frac{1}{p}\sum_{i=1}^{p-1}\binom{p}{i}(Z+1)^i(-1)^{p-i}\right) \\
          =& (Z+1-1)^p + \frac{1}{p}\left[\sum_{j=  0}^{p-1}\left(\sum_{i = j}^{p-1}\binom{i}{j}Y^j\right) - p\right]\cdot\left[\sum_{i=1}^{p-1}\binom{p}{i}(Z+1)^i(-1)^{p-i}\right].
    \end{align*}
    A straightforward induction shows that
    \[
      \sum_{i = j}^{p-1}\binom{i}{j} = \binom{p}{j+1},
    \]
    so that:
    \begin{align*}
        \nn_G^K(X) =& Z^p + \frac{1}{p}\left[\sum_{j=0}^p \binom{p}{j+1}Y^j - p\right]\cdot\left[(Z+1-1)^p - (Z+1)^p +1\right] \\
          =& Z^p + \frac{1}{p}\left[\sum_{j=1}^p \binom{p}{j+1}Y^j\right]\cdot\left[-\sum_{i=1}^{p-1}\binom{p}{i}Z^i\right].
      \end{align*}
      Now recall from the proof of \Cref{normingamma} that $pY, pZ \in \Gamma^{p}(G\times C_p)$. Thus:
      \[
        \nn_G^{G\times C_p}(X) \equiv Z^p - ZY^{p-1} (\Mod \ \Gamma^{p+1}),
      \]
      which concludes the proof.
  \end{itemize}
\end{proof}

\bibliographystyle{alpha}

\bibliography{bibliography}

\end{document}